\documentclass{amsart}
\usepackage{preamble}

\begin{document}

\title[Freeness theorem for cohomology of Rep($C_2$)-complexes]{The freeness theorem for equivariant cohomology of Rep($C_2$)-complexes}
\author{Eric Hogle and Clover May}

\begin{abstract}
Let $C_2$ be the cyclic group of order two.  We show that the $RO(C_2)$-graded Bredon cohomology of a finite $\Rep(C_2)$-complex is free as a module over the cohomology of a point when using coefficients in the constant Mackey functor $\underline{\F_2}$.  This paper corrects some errors in Kronholm's proof of this freeness theorem.  It also extends the freeness result to finite type complexes, those with finitely many cells of each fixed-set dimension.  We give a counterexample showing the theorem does not hold for locally finite complexes.
\end{abstract}

\maketitle

\setcounter{tocdepth}{1}
\tableofcontents

\section{Introduction}

Let $C_2$ be the cyclic group of order two.  In this paper we are concerned with $\Rep(C_2)$-complexes, a class of $C_2$-spaces built using representations.  Our main goal is to correct some subtle errors in Kronholm's proof \cite{K} that finite $\Rep(C_2)$-complexes have free $RO(C_2)$-graded Bredon cohomology in constant $\underline{\F_2}$-coefficients.

Having proved Kronholm's theorem, we go on to extend the result to finite type complexes.  A finite type $\Rep(C_2)$-complex is one with finitely many cells of each fixed-set dimension (and hence also of each topological dimension).  The key to this proof is to filter cohomology by the fixed-set dimension of the generators and argue the vanishing of the $\varprojlim^1$ term.  The freeness theorem cannot be extended to all locally finite complexes, which have finitely many cells of each topological dimension.  A particular infinite wedge of representation spheres is a counterexample.  We show the generalized freeness theorem lifts to a splitting at the spectrum level.

Nonequiviariantly, working with coefficients in the field $\F_2$ makes singular cohomology computations rather straightforward.  For one thing, as an $\F_2$-module, every vector space is free.  Moreover, when calculating the induced map on cohomology for attaching a single cell, one only needs to consider generators from one dimension lower. Computations in $RO(C_2)$-Bredon graded cohomology of $C_2$-spaces are much more challenging, even with coefficients in the equivariant analogue $\underline{\F_2}$.  In fact, the $RO(C_2)$-graded cohomology of a point in $\underline{\F_2}$-coefficients is an infinite-dimensional non-Noetherian ring.  This ring, $\Mt$, has a complicated module theory, making freeness theorems highly nontrivial.  Furthermore, attaching maps for a cell can involve cohomology generators from lower dimensions.

Kronholm's theorem is a powerful computational tool.  It shows the cohomology of a finite $\Rep(C_2)$-complex is free as an $\Mt$-module.  Moreover, it demonstrates freeness even in the presence of nonzero differentials corresponding to the attaching maps for representation cells.  This solves numerous extension problems in computations.  Even so, finding the degrees of free generators is often nontrivial.  Kronholm's freeness theorem has been used by Dugger \cite{DGrass} to study a class of infinite $C_2$-equivariant Grassmannians that are finite type $\Rep(C_2)$-complexes.  It has also been used by the first author \cite{H} to study some families of finite Grassmannians.

Prior to Kronholm's work on the freeness theorem, Lewis \cite{L} proved a freeness theorem for the cohomology of $\Rep(C_p)$-complexes, where $p$ is any prime.  Lewis requires the complexes have only even-dimensional cells with a further restriction on the fixed-set dimensions.  These restrictions force all differentials to be zero in the long exact sequence for attaching a cell.  Ferland \cite{F}, building on the work of Lewis, generalized the freeness theorem to finite type $\Rep(C_p)$-complexes with even-dimensional cells for $p$ odd.  Ferland's result, like Kronholm's, allows for nonzero differentials.  At odd primes, it is not possible to extend Ferland's freeness theorem to include all finite type $\Rep(C_p)$-complexes (see Counterexample \ref{odd prime ex}).  This makes it all the more surprising that Kronholm's freeness theorem holds for all finite type $\Rep(C_2)$-complexes.

The gap in Kronholm's argument occurs during the inductive step, where he implies that we can reduce to the case of a differential supported by a single free summand. In fact this is not always possible.  The mistake appears to arise from the similarity between spectral sequences for two different filtrations of the space.  One spectral sequence comes from the two-stage filtration for attaching a single cell to a complex, while the other is for the `one-at-a-time' cellular filtration. Kronholm's paper appears to conflate these two approaches and incorrectly apply reasoning from one spectral sequence to the other.

There are multiple approaches one might use to correct the proof.  In this paper, we focus on the filtration for attaching a single cell and complete the inductive step.  Here the spectral sequence is really just the long exact sequence associated to a cofiber sequence, and so we will simply refer to this as a long exact sequence throughout the paper. An argument carefully extending Kronholm's techniques to a differential supported by $n$ free summands would likely work, but would require extensive bookkeeping.

To simplify the exposition, our proof uses the second author's recent $C_2$ structure theorem \cite{CM}, together with several localization arguments.  The structure theorem applies more generally to finite $C_2$-CW complexes and says their cohomology can only have two types of direct summands: free modules, and shifted copies of the cohomologies of antipodal spheres.  In the context of this result, we need only show the antipodal spheres do not appear in the cohomology of a $\Rep(C_2)$-complex.  While this may sound simple, it is still rather technical to prove the inductive step.

Somewhat surprisingly, we prove the freeness theorem for cohomology without determining a basis of free generators.  However, in practice, a free basis is useful for computations.  In Section \ref{shifts}, we explicitly compute a basis for the cohomology of a $\Rep(C_2)$-complex in the presence of a nontrivial differential.  Kronholm observed that a nontrivial differential causes generators to appear to ``shift'' from their original positions and we give formulas for these shifts.

\subsection{Organization}  In Section \ref{preliminaries} we introduce the necessary background and notation, largely from \cite{K}.  In Section \ref{comp tools} we recall a number of computational tools from \cite{K} and \cite{CM}.  The main proof will require a change of basis of a free module, the first algebraic steps of which are given in Section \ref{change of basis}.  A further restriction on the change of basis for the cohomology of a space is given in Section \ref{basis for attaching}.
These two steps are similar to Kronholm's change of basis in \cite{K}.
In Section \ref{main section} we give a proof of Kronholm's freeness theorem for finite complexes and extend the result to finite type complexes.  We also show the freeness theorem lifts to a splitting at the spectrum level.  In Section \ref{shifts} we calculate the changes in the degrees of the generators after a nontrivial differential.  We also give an explicit basis.  In Section \ref{Kronholm proof} we explain in more detail the main error in Kronholm's paper, which led to this work.

\subsection{Acknowledgments}  The authors would like to thank Dan Dugger for introducing them to the beautiful subject of equivariant topology and for his support throughout numerous revisions.  Thanks also to Mike Hill for many helpful conversations, particularly regarding the finite type and locally finite cases.  Finally, thank you to the anonymous referee for helpful suggestions.  This work was partially funded by the University of Oregon, UCLA, and Gonzaga University.

\section{Preliminaries}\label{preliminaries}

To begin, we set up some basic machinery and notation much as in \cite{K} and \cite{M} with a few small variations.  Let $G$ be a finite group.  Given an orthogonal real $G$-representation $V$, let $D(V)$ and $S(V)$ denote the unit disk and unit sphere  in $V$, respectively.  Let $S^V = \widehat{V}$ denote the representation sphere given by the one-point compactification of $V$.  There are two important types of equivariant cell complexes.

\begin{defn}
A \mdfn{$G$-CW complex} is a $G$-space $X$ with a filtration, where $X_0$ is a disjoint union of orbits $G/H$ and $X_{n}$ is obtained from $X_{n-1}$ by attaching cells of the form $(G/H_\alpha) \times D^n$ along equivariant maps $f_\alpha: G/H_\alpha \times S^{n-1} \to X_{n-1}$.  The cells are attached via the usual pushout diagram
\begin{center}
\begin{tikzcd}
\coprod_{\alpha} G/H_\alpha \times S^{n-1} \arrow{r}{\sqcup_\alpha f_\alpha} \arrow[hook, d] & X_{n-1} \arrow[d] \\
\coprod_{\alpha} G/H_\alpha \times D^n \arrow[r] & X_n
\end{tikzcd}
\end{center}
where $D^n$ and $S^{n-1}$ have the trivial $G$-action.
\end{defn}

We will mainly be interested in another type of cell structure, one that is built with representation cells, called a $\Rep(G)$-complex.
\begin{defn}\label{Rep complex def}
A \mdfn{$\Rep(G)$-complex} is a $G$-space $X$ with a filtration $X_{n}$ where $X_{0}$ is a disjoint union of trivial orbits\footnote{Note that the definition in \cite{K} allows for $X_{0}$ to be made up of any $G$-orbits.  However, $C_2$ is itself a $C_2$ orbit and does not have free cohomology, which would contradict the freeness theorem.} of the form $G/G = *$ and $X_{n}$ is obtained from $X_{n-1}$ by attaching cells of the form $D(V_\alpha)$,
where $V_\alpha$ is an $n$-dimensional real representation of $G$.  The cells are attached along maps $f_\alpha : S(V_\alpha) \to X_{n-1}$ via the usual pushout diagram.
\end{defn}

The space $X_{n}$ in either filtration is referred to as the \mdfn{$n$-skeleton} of $X$ and the filtration is referred to as a \dfn{cell structure}.
If the filtration is finite, then $X$ is \dfn{finite dimensional}.  If there are finitely many cells of each dimension, then $X$ is called \dfn{locally finite}.  We call a $\Rep(C_2)$-complex \dfn{finite type} if it has finitely many cells of each fixed-set dimension, as defined below.  If $X$ is a connected $\Rep(G)$-complex, the filtration quotients are wedges of representation spheres $X_{n}/X_{n-1} \cong \bigvee_\alpha S^{V_\alpha}$.

\begin{remark}
Any $\Rep(G)$-complex can be given the structure of a $G$-CW complex.  The converse is false.  In particular, any $G$-space with a free action cannot be given the structure of a $\Rep(G)$-complex.  A $\Rep(G)$-complex has at least one fixed point because the origin of any real representation is fixed.
\end{remark}

We now specialize to the group $G=C_2$.  As in \cite{K}, we write a $p$-dimensional real $C_2$-representation $V$ as
\[
V \cong (\R^{1,0})^{p-q} \oplus (\R^{1,1})^{q} = \R^{p,q}
\]
where $\R^{1,0}$ is the trivial $1$-dimensional real representation of $C_2$ and $\R^{1,1}$ is the sign representation.  Allowing $p$ and $q$ to be integers if $V$ is a virtual representation, we call $p$ the \dfn{topological dimension} and $q$ the \dfn{weight} or \dfn{twisted dimension} of $V = \R^{p,q}$.  We will also refer to the \dfn{fixed-set dimension}, which is $p-q$. This is also referred to in the literature as \dfn{coweight}.  We write $S^V = S^{p,q}$ for the (possibly virtual) \dfn{representation sphere} given by the one-point compactification of $V$.

For the $V$-th graded component of the ordinary $RO(C_2)$-graded Bredon equivariant cohomology of a $C_2$-space $X$ with coefficients in the constant Mackey functor $\underline{\F_2}$, we write $H^V_G(X;\underline{\F_2}) = H^{p,q}(X;\underline{\F_2})$.  We often suppress the coefficients and simply write $H^{p,q}(X;\underline{\F_2}) = H^{p,q}(X)$.  When we work nonequivariantly, $H^*_{\sing}(X)$ denotes the singular cohomology of the underlying topological space $X$ with $\F_2$-coefficients.
The genuine equivariant Eilenberg--MacLane spectrum representing $\tilde{H}^{*,*}(-)$ is $H\underline{\F_2}$.  It has as its underlying spectrum $H\F_2$.
Given a homogeneous element $x \in H^{p,q}(X)$, we use the notation $|x| = (p,q)$ for the bidegree, $\Top(x) = p$ for the topological dimension, $\wt(x) = q$ for the weight, and $\fix(x) = p - q$ for the fixed-set dimension.\footnote{This is a departure from the usual notation. Kronholm \cite{K}, Shulman \cite{S}, and Ferland and Lewis \cite{FL} use the notation $|x|$ to denote the topological dimension $p$ rather than the bidegree, and the notation $|x^G|$ to denote the fixed-set dimension $p-q$.}  It is often convenient to plot the bigraded cohomology in the plane.  We will always plot the topological dimension $p$ horizontally and the weight $q$ vertically.

With coefficients in the constant Mackey functor $\underline{\F_2}$, the cohomology of a point with the trivial $C_2$-action is the ring $\Mt := H^{*,*}(\pt)$ pictured in Figure \ref{M2}.  On the left is a more detailed depiction, however in practice it is easier to work with the more succinct version on the right.  Every lattice point inside the two ``cones" represents a copy of $\F_2$.  There are unique nonzero elements $\rho \in H^{1,1}(\pt)$ and $\tau \in H^{0,1}(\pt)$.
Considered as an $\F_2[\rho,\tau]$-module, $\Mt$ splits as $\Mt = \Mt^+ \oplus \Mt^-$ where the top cone $\Mt^+$ is a polynomial algebra with generators $\rho$ and $\tau$.  There is a unique nonzero element in bidegree $(0,-2)$ of the bottom cone $\Mt^-$. This element $\theta \in H^{0,-2}(\pt)$ is infinitely divisible by both $\rho$ and $\tau$ and satisfies $\theta^2 = 0$.
We say that every element of the lower cone is both $\rho$-torsion and $\tau$-torsion because it is killed by a multiple of $\rho$ and some multiple of $\tau$.

\begin{figure}[ht]
\begin{center}\hfill
\begin{tikzpicture}[scale=0.6]
\draw[gray] (-3.5,0) -- (4.5,0) node[below,black] {\small $p$};
\draw[gray] (0,-4.5) -- (0,4.5) node[left,black] {\small $q$};
\foreach \x in {-3,...,-1,1,2,...,4}
	\draw [font=\tiny, gray] (\x cm,2pt) -- (\x cm,-2pt) node[anchor=north] {$\x$};
\foreach \y in {-4,...,-1,1,2,...,4}
	\draw [font=\tiny, gray] (2pt,\y cm) -- (-2pt,\y cm) node[anchor=east] {$\y$};

\foreach \y in {0,...,4}
	\fill (0,\y) circle(2.5pt);
\foreach \y in {1,...,4}
	\fill (1,\y) circle(2.5pt);
\foreach \y in {2,...,4}
	\fill (2,\y) circle(2.5pt);
\foreach \y in {3,...,4}
	\fill (3,\y) circle(2.5pt);
\foreach \y in {4,...,4}
	\fill (4,\y) circle(2.5pt);

\foreach \y in {0,...,2}
	\fill (0,-\y-2) circle(2.5pt);
\foreach \y in {1,...,2}
	\fill (-1,-\y-2) circle(2.5pt);
\foreach \y in {2,...,2}
	\fill (-2,-\y-2) circle(2.5pt);

\draw[thick,->] (0,0) -- (4.5,4.5);
\draw[thick,->] (0,0) -- (0,4.5);
\draw[thick,->] (0,-2) -- (0,-4.5);
\draw[thick,->] (0,-2) -- (-2.5,-4.5);

\draw (0,-0.3) node[below,right] {$1$};
\draw (1,1) node[right] {$\rho$};
\draw (0,1) node[right] {$\tau$};
\draw (0,-2) node[right] {$\theta$};
\draw (-1.1,-3) node[left] {$\frac{\theta}{\rho}$};
\draw (0,-3) node[right] {$\frac{\theta}{\tau}$};

\end{tikzpicture} \hfill
\begin{tikzpicture}[scale=0.6]
\draw[gray] (-3.5,0) -- (4.5,0) node[below,black] {\small $p$};
\draw[gray] (0,-4.5) -- (0,4.5) node[left,black] {\small $q$};
\foreach \x in {-3,...,-1,1,2,...,4}
	\draw [font=\tiny, gray] (\x cm,2pt) -- (\x cm,-2pt) node[anchor=north] {$\x$};
\foreach \y in {-4,...,-1,1,2,...,4}
	\draw [font=\tiny, gray] (2pt,\y cm) -- (-2pt,\y cm) node[anchor=east] {$\y$};
\draw[gray] (2pt, -1cm) -- (-2pt,-1cm);

\draw[thick] (0,0) -- (4.5,4.5);
\draw[thick] (0,0) -- (0,4.5);
\draw[thick] (0,-2) -- (0,-4.5);
\draw[thick] (0,-2) -- (-2.5,-4.5);

\fill (0,0) circle(2.5pt);
\draw[transparent] (5.5,0) node{$M2$};

\end{tikzpicture}
\end{center}
\caption{$\Mt = H^{*,*}(\pt;\underline{\F_2})$.}
\label{M2}
\end{figure}

The $RO(C_2)$-graded cohomology $H^{*,*}(X)$ is a bigraded $\Mt$-module.  By $\Mt$-module we always mean bigraded $\Mt$-module, and any reference to an $\Mt$-module map means a bigraded homomorphism. For a free $\Mt$-module with a single generator $\omega$ with bidegree $|\omega| = (p,q)$ we use the notation $\Mtb{\omega} = \Sigma^{p,q}\Mt$.

We write $\A_n$ for the cohomology of $S^n_a$, the $n$-dimensional sphere with the antipodal action, as an $\Mt$-module.  Notice that $S^n_a$ has a free $C_2$-action, so this is not an example of $\Rep(C_2)$-complex.  The $\Mt$-module $\A_n$ plays an important role in the cohomology of $C_2$-CW complexes, though we will see that it cannot appear in the cohomology of a $\Rep(C_2)$-complex.

A picture of $\A_n$ (actually of $\A_4$) appears in Figure \ref{An}.  Once again, on the left is a more detailed depiction, while in practice it is more convenient to use the succinct version on the right.  Here every lattice point in the infinite strip of width $n+1$ represents an $\F_2$.  Diagonal lines represent multiplication by $\rho$ and vertical lines represent multiplication by $\tau$, so that every nonzero element in $\A_n$ is in the image of $\tau$ and is not $\tau$-torsion.
We allow for $n = 0$ since $C_2 = S^0_a$ has cohomology given by a single vertical line.  As a ring $\A_n \cong \F_2[\tau,\tau^{-1},\rho]/(\rho^{n+1})$, where multiplication by $\rho$ and $\tau$ corresponds to the module multiplication by the usual elements in $\Mt$ and where $\tau^{-1}$ has bidegree $(0,-1)$.

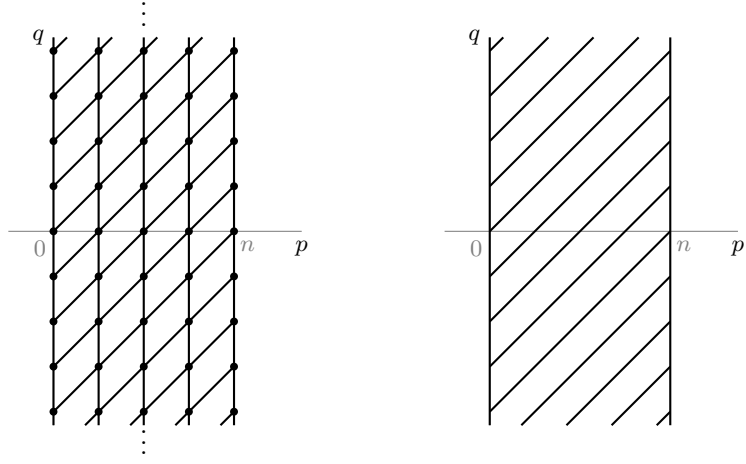
\begin{figure}[ht]
\begin{center}\hfill
\begin{tikzpicture}[scale=0.6]
\draw[gray] (-1,0) -- (5.5,0) node[below,black] {\small $p$};
\draw[gray] (0,-4.3) -- (0,4.3) node[left,black] {\small $q$};
\draw [font=\small, gray] (-0.3,0) node[below] {$0$};
\draw [font=\small, gray] (4.3,0) node[below] {$n$};

\draw (2,5) node {$\vdots$};
\foreach \x in {0,...,4}
	\foreach \y in {-4,...,4}
		\fill (\x,\y) circle(2.5pt);

\foreach \x in {0,...,4}
	\draw[thick] (\x,-4.3) -- (\x,4.3);
\foreach \y in {-4,...,0}
	\draw[thick] (0,\y) -- (4,\y+4);

\draw[thick] (0,1) -- (3.3,4.3);
\draw[thick] (0,2) -- (2.3,4.3);
\draw[thick] (0,3) -- (1.3,4.3);
\draw[thick] (0,4) -- (0.3,4.3);

\draw[thick] (0.7,-4.3) -- (4,-1);
\draw[thick] (1.7,-4.3) -- (4,-2);
\draw[thick] (2.7,-4.3) -- (4,-3);
\draw[thick] (3.7,-4.3) -- (4,-4);
\draw[thick] (2,-4.5) node {$\vdots$};

\end{tikzpicture}\hfill
\begin{tikzpicture}[scale=0.6]
\draw[gray] (-1,0) -- (5.5,0) node[below,black] {\small $p$};
\draw[gray] (0,-4.3) -- (0,4.3) node[left,black] {\small $q$};
\draw [font=\small, gray] (-0.3,0) node[below] {$0$};
\draw [font=\small, gray] (4.3,0) node[below] {$n$};

\draw[transparent] (2,5) node {$\vdots$};

\foreach \x in {0,4}
	\draw[thick] (\x,-4.3) -- (\x,4.3);
\foreach \y in {-4,...,0}
	\draw[thick] (0,\y) -- (4,\y+4);

\draw[thick] (0,1) -- (3.3,4.3);
\draw[thick] (0,2) -- (2.3,4.3);
\draw[thick] (0,3) -- (1.3,4.3);
\draw[thick] (0,4) -- (0.3,4.3);

\draw[thick] (0.7,-4.3) -- (4,-1);
\draw[thick] (1.7,-4.3) -- (4,-2);
\draw[thick] (2.7,-4.3) -- (4,-3);
\draw[thick] (3.7,-4.3) -- (4,-4);
\draw[transparent] (2,-4.5) node {$\vdots$};

\draw[transparent] (7,0) node {An};
\end{tikzpicture}
\end{center}
\caption{$\A_n = H^{*,*}(S^n_a;\underline{\F_2})$.}
\label{An}
\end{figure}

\section{Computational tools}\label{comp tools}

In this section we present some tools for computing the $RO(C_2)$-graded cohomology of $C_2$-equivariant spaces.  Let $X$ be a connected $\Rep(C_2)$-complex.  Then $X$ has a filtration coming from the cell structure where the filtration quotients $X_{n}/X_{n-1}$ are wedges of $n$-dimensional representation spheres corresponding to the representation cells that were attached.

More generally, given any filtration of a $C_2$-space $X$
\[
\pt = X_0 \subseteq X_1 \subseteq \cdots \subseteq X_k \subseteq X_{k+1} \subseteq \cdots \Rightarrow X
\]
corresponding to each cofiber sequence
\[
X_{k} \hookrightarrow X_{k+1} \to X_{k+1}/X_{k}
\]
and for each weight $q$ there is a long exact sequence\footnote{As Kronholm \cite{K} points out, these long exact sequences sew together in the usual way to give a spectral sequence.  See Proposition \ref{spec seq} in Section \ref{Kronholm proof}.}
\[
\cdots \to \tilde{H}^{p,q}(X_{k+1}/X_{k}) \to \tilde{H}^{p,q}(X_{k+1}) \to \tilde{H}^{p,q}(X_{k}) \xrightarrow{d} \tilde{H}^{p+1,q}(X_{k+1}/X_{k})\to \cdots.
\]
We often refer to the long exact sequences taken collectively for all $q$ as ``the long exact sequence."  Then $d$, taken collectively for all $p$ and $q$, is a graded $\Mt$-module map $d: \tilde{H}^{*,*}(X_{k}) \to \tilde{H}^{*+1,*}(X_{k+1}/X_{k})$, which we call the ``differential'' in the long exact sequence.  This gives a short exact sequence of graded $\Mt$-modules
\[
0 \to \cok (d) \to \tilde{H}^{*,*}(X_{k+1}) \to \ker (d) \to 0.
\]
In many cases $\cok(d)$ and $\ker(d)$ are relatively easily determined, but computing $\tilde{H}^{*,*}(X_{k+1})$ requires solving the extension problem presented in this short exact sequence.

As in the previous section, we plot $RO(C_2)$-graded cohomology in the plane with the topological dimension $p$ along the horizontal axis and the weight $q$ along the vertical axis.  The differential $d$ in the long exact sequence is depicted by a horizontal arrow since it increases topological dimension by one.  When $\tilde{H}^{*,*}(X_{k})$ is free as a graded $\Mt$-module, i.e.\! when
\[
\tilde{H}^{*,*}(X_{k}) \cong \Mtb{\gamma_1, \dots, \gamma_k} = \bigoplus_{i} \Sigma^{|\gamma_i|} \Mt,
\]
the differential is determined by its image $d(\gamma_i)$ on the basis elements or on any set of generators.

Before moving on we present an example that illustrates some common computational techniques as well as some advantages of the main theorem.  In the computation presented here we use the following fact from Section 6 in \cite{CM}, which says we can compute the $p$-axis of the $RO(C_2)$-graded cohomology of a space using singular cohomology of the quotient.
\begin{lemma} \label{quotient lemma}
Let $X$ be a $C_2$-space.  Then $\tilde{H}^{p,0}(X) \cong H^p_{\sing}(X/C_2)$.
\end{lemma}

\begin{example}\label{RP^2}
In this example we compute the cohomology of the projective space $\R P^2_{tw} = \mathbb{P}(\R^{3,1})$ using Lemma \ref{quotient lemma}.  A picture of $\R P^2_{tw}$ is shown in Figure \ref{real proj fig}.  This is the usual depiction of a disk with opposite points on the boundary identified.  The $C_2$-action is given by rotating the picture $180^\circ$.

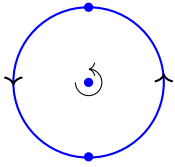
\begin{figure}[ht]
\begin{center}
\begin{tikzpicture}[scale=1]
\draw[thick, fixedgreen,-] (0,0) arc (-90:90:1cm);
\draw[thick,->] (1,1.1) -- (1,1.12);
\draw[thick, fixedgreen,-] (0,2) arc (90:270:1cm);
\draw[thick,->] (-1,0.98) -- (-1,0.96);
\draw[->] (-0.175,1) arc (-180:90:5pt);
\fill[fixedgreen] (0,0.01) circle(1.75pt);
\fill[fixedgreen] (0,2) circle(1.75pt);
\fill[fixedgreen] (0,1) circle(1.75pt);

\draw[transparent] (1,0) node {stuff};
\end{tikzpicture}
\end{center}
\caption{A depiction of $\R P^2_{tw}$.}
\label{real proj fig}
\end{figure}

The long exact sequence associated to the cofiber sequence $S^{1,0} \hookrightarrow \R P^2_{tw} \to S^{2,2}$ is depicted on the left side of Figure \ref{differential real proj}.  Recall that in these depictions every lattice point inside the cones represents an $\F_2$.

The map $d$ is determined by its image on the generator of $\tilde{H}^{*,*}(S^{1,0}) \cong  \Sigma^{1,0}\Mt$.  It is necessarily nonzero because the quotient space $\R P^2_{tw}/C_2$ is the cone on $S^1$, which is contractible.  The modules ${\color{blue}\cok(d)}$ and ${\color{red}\ker(d)}$ resulting from this differential are on the right side of Figure \ref{differential real proj}.   Even knowing the differential, computing $\tilde{H}^{*,*}(\R P^2_{tw})$ requires solving the extension problem in the short exact sequence
\[
0 \to \color{blue}{\cok(d)}\, {\color{black}\to \tilde{H}^{*,*}(\R P^2_{tw}) \to} {\,\color{red}\ker(d)}\, \color{black}{\to 0.}
\]
It is not at all obvious at this stage that the solution to this extension problem should be a free $\Mt$-module.  However, since $\R P^2_{tw}$ is a $\Rep(C_2)$-complex $\tilde{H}^{*,*}(\R P^2_{tw})$ must be free as a result of the freeness theorem in Section \ref{main thm}.  Thus the cohomology of $\R P^2_{tw}$ as an $\Mt$-module is $\tilde{H}^{*,*}(\R P^2_{tw}) = \Sigma^{1,1}\Mt \oplus \Sigma^{2,1}\Mt$, as pictured in Figure \ref{cohomology}.

\begin{figure}[ht]
\begin{center}\hfill
\begin{tikzpicture}[scale=0.6]
\draw[gray] (-3.5,0) -- (4.5,0) node[below,black] {\small $p$};
\draw[gray] (0,-4.5) -- (0,4.5) node[left,black] {\small $q$};
\foreach \x in {-3,...,-1,1,2,...,4}
	\draw [font=\tiny, gray] (\x cm,2pt) -- (\x cm,-2pt) node[anchor=north] {$\x$};
\foreach \y in {-4,...,-1,1,2,...,4}
	\draw [font=\tiny, gray] (2pt,\y cm) -- (-2pt,\y cm) node[anchor=east] {$\y$};

\draw[->,thick,black] (1,0) -- (1.95,0);
\draw (1.6,0) node[above] {\small $d$};

\draw[thick,red] (1,0) -- (4.5,3.5);
\draw[thick,red] (1,0) -- (1,4.5);
\draw[thick,red] (1,-2) -- (1,-4.5);
\draw[thick,red] (1,-2) -- (-1.5,-4.5);
\fill[red] (1,0) circle(2.5pt);

\draw[thick,blue] (2,2) -- (4.5,4.5);
\draw[thick,blue] (2,2) -- (2,4.5);
\draw[thick,blue] (2,0) -- (2,-4.5);
\draw[thick,blue] (2,0) -- (-2.5,-4.5);
\fill[blue] (2,2) circle(2.5pt);

\end{tikzpicture}\hfill
\begin{tikzpicture}[scale=0.6]
\draw[gray] (-3.5,0) -- (4.5,0) node[below,black] {\small $p$};
\draw[gray] (0,-4.5) -- (0,4.5) node[left,black] {\small $q$};
\foreach \x in {-3,...,-1,1,2,...,4}
	\draw [font=\tiny, gray] (\x cm,2pt) -- (\x cm,-2pt) node[anchor=north] {$\x$};
\foreach \y in {-4,...,-1,1,2,...,4}
	\draw [font=\tiny, gray] (2pt,\y cm) -- (-2pt,\y cm) node[anchor=east] {$\y$};

\draw[thick,red] (1,1) -- (2,2) -- (2,1) -- (4.5,3.5);
\draw[thick,red] (1,1) -- (1,4.5);
\draw[thick,red] (1,-2) -- (1,-4.5);
\draw[thick,red] (1,-2) -- (-1.5,-4.5);
\draw[red] (3.5,1.5) node {\small $\ker(d)$};

\draw[thick,blue] (2,2.1) -- (4.5,4.6);
\draw[thick,blue] (2,2.1) -- (2,4.5);
\draw[thick,blue] (2,-0.9) -- (2,-4.5);
\draw[thick,blue] (2,-0.9) -- (1,-1.9) -- (1,-0.9) -- (-2.6,-4.5);
\draw[blue] (3,-2) node {\small $\cok(d)$};

\draw[transparent] (5.5,0) node {RP2};
\end{tikzpicture}
\end{center}
\caption{Differential in a long exact sequence for $\tilde{H}^{*,*}(\R P^2_{tw}$).}\label{differential real proj}
\end{figure}
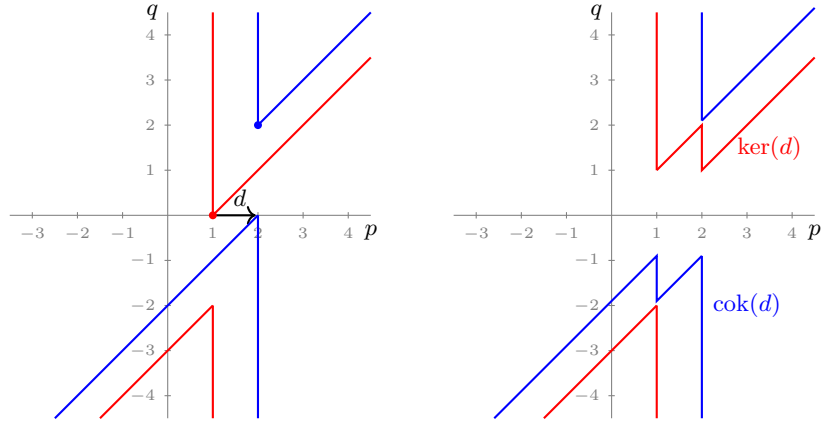

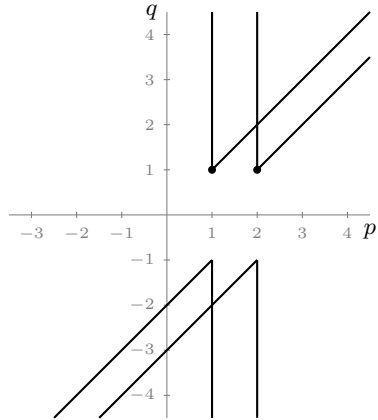
\begin{figure}[ht]
\begin{center}
\begin{tikzpicture}[scale=0.6]
\draw[gray] (-3.5,0) -- (4.5,0) node[below,black] {\small $p$};
\draw[gray] (0,-4.5) -- (0,4.5) node[left,black] {\small $q$};
\foreach \x in {-3,...,-1,1,2,...,4}
	\draw [font=\tiny, gray] (\x cm,2pt) -- (\x cm,-2pt) node[anchor=north] {$\x$};
\foreach \y in {-4,...,-1,1,2,...,4}
	\draw [font=\tiny, gray] (2pt,\y cm) -- (-2pt,\y cm) node[anchor=east] {$\y$};

\draw[thick] (1,1) -- (4.5,4.5);
\draw[thick] (1,1) -- (1,4.5);
\draw[thick] (1,-1) -- (1,-4.5);
\draw[thick] (1,-1) -- (-2.5,-4.5);
\fill (1,1) circle(2.5pt);

\draw[thick] (2,1) -- (4.5,3.5);
\draw[thick] (2,1) -- (2,4.5);
\draw[thick] (2,-1) -- (2,-4.5);
\draw[thick] (2,-1) -- (-1.5,-4.5);
\fill (2,1) circle(2.5pt);

\draw[transparent] (5.5,0) node {RP2};
\end{tikzpicture}
\end{center}
\caption{Reduced cohomology of $\R P^2_{tw}$.}
\label{cohomology}
\end{figure}

Notice that in the cohomology of $\R P^2_{tw}$ there are two copies of $\Mt$ generated in the same topological dimensions as before, but they have ``shifted."  One generator now has higher weight and the other lower weight.  This is an example of a more general behavior known as a ``Kronholm shift'', described in Section \ref{shifts}.  We give formulas that precisely quantify these shifts in Theorem \ref{formula}.  The power of the freeness theorem and the shifting formulas is that when a nonzero differential like the one above occurs, the resulting cohomology must be free and the bidegrees of the free generators are determined.
\end{example}

\begin{aside}  One might expect a similar freeness result to hold more generally for $\Rep(C_p)$-complexes.  However, for odd primes one quickly discovers a space analogous to $\R P^2_{tw}$ that does not have free cohomology.  This makes it rather surprising that the freeness theorem holds for all finite $\Rep(C_2)$-complexes.  For concreteness, we take $p=3$ and give an example below of a $\Rep(C_3)$-complex whose cohomology is \emph{not} free.

\begin{counterexample}\label{odd prime ex}
Let $X$ be the $C_3$-space whose underlying space is a $2$-simplex with edges identified as pictured in Figure \ref{odd prime fig}.  Here a generator of $C_3$ acts by rotating the picture $120^\circ$.

\begin{figure}[ht]
\begin{center}
\begin{tikzpicture}[scale=1]
\draw[thick,fixedgreen,-] (-1.5,0.3) -- (1.5,0.3) -- (0,2.6) -- cycle;
\draw[thick,->] (-0.83,1.325) -- (-0.84,1.31);
\draw[thick,->] (0.765,1.43) -- (0.755,1.45);
\draw[thick,->] (0.03,0.3) -- (0.05,0.3);
\draw[->] (-0.175,1.15) arc (-180:90:5pt);
\fill[fixedgreen] (-1.5,0.3) circle(1.5pt);
\fill[fixedgreen] (0,2.6) circle(1.5pt);
\fill[fixedgreen] (0,1.15) circle(1.25pt);
\fill[fixedgreen] (1.5,0.3) circle(1.5pt);

\draw[transparent] (1.25,0) node {stuff};
\end{tikzpicture}
\end{center}
\caption{A $C_3$ analogue of $\R P^2_{tw}$.}
\label{odd prime fig}
\end{figure}
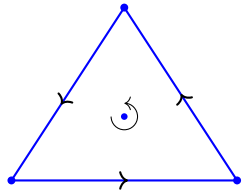

The space $X$ can be realized as a $\Rep(C_3)$-complex.  Using techniques similar to those in Example \ref{RP^2}, one can compute the cohomology of $X$ with constant $\underline{\F_3}$-coefficients as a module over the cohomology of a point.
Of course, the cohomology of a point as a $C_3$-space with $\underline{\F_3}$ coefficients is not $\Mt$, but it shares several properties with $\Mt$ and can also be depicted with two cones.
One can readily verify the cohomology of $X$ is not free as a module over this ring.
\end{counterexample}
\end{aside}

We now return to the prime two and assemble a few more computational tools.  We will use several results from \cite{CM} to simplify the proof of the freeness theorem.  In particular, from \cite{CM} we have the following structure theorem for the $RO(C_2)$-graded cohomology of $C_2$-CW complexes.  The structure theorem says that as a module over the cohomology of the point, the $RO(C_2)$-graded cohomology of a finite $C_2$-CW complex decomposes as a direct sum of two basic pieces: shifted copies of the cohomology of a point and shifted copies of the cohomologies of spheres with the antipodal action.

\begin{thm}[Structure Theorem] \label{structure thm}
For any finite $C_2$-CW complex $X$, there is a decomposition of the $RO(C_2)$-graded cohomology of $X$ as an $\Mt$-module given by
\[
H^{*,*}(X;\underline{\F_2}) \cong \left(\bigoplus_i \Sigma^{p_i,q_i} \Mt\right) \oplus \left(\bigoplus_j \Sigma^{r_j,0} \A_{n_j}\right),
\]
where $0 \leq q_i \leq p_i$, $0 \le r_j$, and $0 \le n_j$.
\end{thm}

The goal of this paper is to show that for the special case of a finite $\Rep(C_2)$-complex, the cohomology is free.  That is, we will show the cohomology contains only shifted copies of $\Mt$ and not any copies of $\A_n$.

From the structure theorem we immediately obtain a description of the localizations of the cohomology of a finite $C_2$-CW complex.  Notice that $\A_n$ is preserved by $\tau$-localization while $\tau^{-1}\Mt \cong \A_\infty \cong \F_2[\tau^{\pm 1},\rho]$.
On the other hand $\rho$-localization kills $\A_n$ and $\rho^{-1}\Mt \cong \F_2[\tau,\rho^{\pm 1}]$.

\begin{cor}
Let $X$ be a finite $C_2$-CW complex with
\[
H^{*,*}(X) \cong \left(\bigoplus_i \Sigma^{p_i,q_i} \Mt\right) \oplus \left(\bigoplus_j \Sigma^{r_j,0} \A_{n_j}\right).
\]
Then
\[
\tau^{-1}H^{*,*}(X) \cong \left(\bigoplus_i \Sigma^{p_i,0} \A_\infty\right) \oplus \left(\bigoplus_j \Sigma^{r_j,0} \A_{n_j}\right)
\]
and
\[
\rho^{-1}H^{*,*}(X) \cong \bigoplus_i \Sigma^{p_i-q_i,0} \left(\rho^{-1}\Mt\right).
\]
\end{cor}

\begin{remark}\label{rho-torsion remark}
Thus for $X$ a finite complex, $H^{*,*}(X)$ is free if and only if $\tau^{-1}H^{*,*}(X)$ has no $\rho$-torsion.  Furthermore, $\rho^{-1}H^{*,*}(X)$ depends only on the fixed-set dimensions of the free generators in $H^{*,*}(X)$.
\end{remark}

In addition to knowing the cohomology is free, we would like to know where the free generators live.  The cohomology of any finite $C_2$-CW complex has ``vanishing regions" where the cohomology is zero.  The following proposition and corollary appear in \cite{CM}.

\begin{prop}\label{vanishing}
Let $X$ be a finite $C_2$-CW complex of dimension $m$.   Then $H^{p,q}(X) = 0$ whenever
\begin{enumerate}
\item[$(1)$] $p < 0$ and $q > p-2$, or
\item[$(2)$] $p > m$ and $q < p-m$.
\end{enumerate}
\end{prop}

In particular, we have the immediate corollary.
\begin{cor}\label{vanishing cor}
Any generator of a copy of $\Mt$ in the cohomology of a finite $m$-dimensional $C_2$-CW complex $X$ must lie in a bidegree $(p,q)$ corresponding to an actual representation with topological dimension $p$ satisfying $0 \leq p \leq m$ and weight $0 \leq q \leq p$.
\end{cor}

The region where $\Mt$ generators can lie is depicted by the triangle in Figure \ref{vanishing fig} on the right.

\begin{figure}[ht]
\begin{tikzpicture}[scale=0.375]

\draw[gray] (-4.5,0) -- (9.5,0);
\draw[gray] (0,-6.5) -- (0,5.5);

\fill[lightgray,opacity=0.25] (0,5.5) -- (0,-2) -- (-4.5,-6.5) -- (-4.5,5.5) -- cycle;
\draw[thick,gray,dashed] (0,5.5) -- (0,-2) -- (-4.5,-6.5);

\fill[lightgray,opacity=0.25] (4,-6.5) -- (4,0) -- (9.5,5.5) -- (9.5,-6.5) -- cycle;
\draw[thick,gray,dashed] (4,-6.5) -- (4,0) -- (9.5,5.5);

\fill[gray] (0,0) -- (4,0) -- (4,4) -- cycle;
\draw[thick] (0,0) -- (4,0) -- (4,4) -- cycle;

\draw (0.1,-2) -- (-0.1,-2) node[left] {\tiny $-2$};
\draw (9.5,0) node[below,black] {\small $p$};
\draw (0,5.5) node[left,black] {\small $q$};
\draw (0.1,4) -- (-0.1,4) node[left] {\tiny $m$};
\draw (4,0) node[below right] {\tiny $m$};

\end{tikzpicture}
\caption{Vanishing regions and region containing $\Mt$ generators.}\label{vanishing fig}
\end{figure}
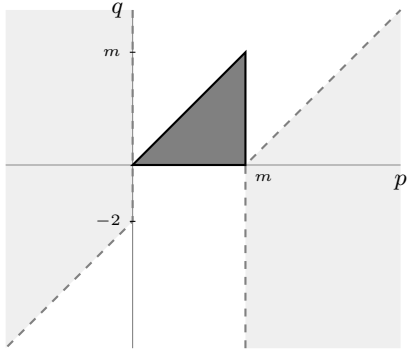

To find $\Mt$ generators we may use $\theta$ as in the following lemma from \cite{CM}.

\begin{lemma}\label{theta lemma}
If a graded $\Mt$-module contains a nonzero homogenous element $x$ with $\theta x$ nonzero, then $\Mtb{x}$ is a graded free submodule.
\end{lemma}

In \cite{CM} it is also shown that $\Mt$ is self-injective.

\begin{lemma}\label{injectivity lemma}
The regular module $\Mt$ is injective as a graded $\Mt$-module.
\end{lemma}

Thus to find free summands of an $\Mt$-module, it is often useful to find an element with a nontrivial $\theta$-multiple.  Such an element generates a free submodule, which splits off as a direct summand because $\Mt$ is self-injective.

\section{Change of basis for free modules}\label{change of basis}
Following in Kronholm's footsteps, in the proof of the freeness theorem in Section \ref{main section}, we will induct on the number of cells of a $\Rep(C_2)$-complex and attach one cell at a time.  For the inductive step, we will need to consider differentials from a free module to a single shifted $\Mt$ corresponding to the newly attached cell.  Kronholm's paper includes a change of basis lemma that simplifies the differentials in this setting \cite[Lemma 3.1]{K}.  However, there is a small error\footnote{A more significant error in Kronholm's main proof will be explained in Section \ref{Kronholm proof}.} in the proof of this lemma.  For completeness, we will first prove two algebraic change of basis lemmas inspired by Kronholm's argument.  In Section \ref{basis for attaching}, we use these algebraic results and a fact about $\rho$-localization to easily deduce the change of basis lemma that appears in \cite{K}.

The hypotheses of both lemmas in this section include restrictions on the topological dimensions and weights of the generators of the free module and its target.  These same constraints will appear in the change of basis lemma in Section \ref{basis for attaching}.  They appear again in the inductive step of the main theorem and are due to the ordering of the cells attached.

For the first change of basis, we consider a free module $\Gamma$ supporting a nonzero map $d$ to the top cone of $\Sigma^{p,q}\Mt = \Mtb{\nu}$.  We will show there is a change of basis for $\Gamma$ so that only one of the basis elements $\lambda$ supports a nonzero map.  The element $\lambda$ will map to $\Mt^+\langle \nu \rangle$.  An example of such a map is depicted in Figure \ref{top cone figure}.  There is a single arrow shown in Figure \ref{top cone figure} because the $\Mt$-module map $d$ is determined by its image on the generator $\lambda$.  However, the map is nonzero in infinitely many other bidegrees.
For example, if $d(\lambda)=\tau \nu$ as depicted, then $d(\tau \lambda) = \tau^2 \nu$, $d\left(\frac{\theta}{\tau} \nu \right) = \theta \nu$, and so on.

\begin{figure}[ht]
\begin{center}
\begin{tikzpicture}[scale=0.6]
\draw[gray] (-2.5,0) -- (5.5,0);
\draw[gray] (0,-2.5) -- (0,7.5);

\draw (2,0.1) -- (2,-0.1) node[below] {\small $p$};
\draw (0.1,2) -- (-0.1,2) node[left] {\small $q$};

\draw[->,thick] (1,3) -- (1.95,3);
\draw (1.6,3) node[above] {\small $d$};

\draw[thick, red] (1,3) -- (5.5,7.5);
\draw[thick, red] (1,3) -- (1,7.5);
\draw[thick, red] (1,1) -- (1,-2.5);
\draw[thick, red] (1,1) -- (-2.5,-2.5);
\draw[red] (1,3) node[below] {$\lambda$};
\fill[red] (1,3) circle(2.5pt);

\draw[thick, blue] (2,2) -- (5.5,5.5);
\draw[thick, blue] (2,2) -- (2,7.5);
\draw[thick, blue] (2,0) -- (2,-2.5);
\draw[thick, blue] (2,0) -- (-0.5,-2.5);
\draw[blue] (2,2) node[below] {$\nu$};
\fill[blue] (2,2) circle(2.5pt);

\end{tikzpicture}
\end{center}
\caption{A single generator $\lambda$ mapping to the upper cone of $\Mtb{\nu}$.}
\label{top cone figure}
\end{figure}
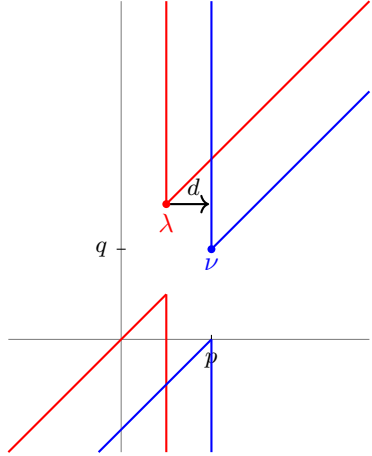

\begin{lemma}\label{top cone lemma}
Consider a nonzero graded $\Mt$-module homomorphism with bidegree $(1,0)$ of the form $\Gamma = \Mtb{ \gamma_0, \gamma_1, \dots, \gamma_{m}} \xrightarrow{d} \Mtb{\nu}$ where $|\nu| = (p,q)$.  Assume for all $i$ that $\Top(\gamma_i) \leq p$ and whenever $\Top(\gamma_i) = p$ that $\wt(\gamma_i) \leq q$.

Suppose $\im (d)$ has a nonzero value in $\Mt^+\langle \nu \rangle$. Then there is a change of basis for $\Gamma$ so that
\[
\Gamma \cong \Mtb{\lambda, \chi_1, \dots, \chi_m},
\]
where
	\begin{itemize}
	\item $d(\lambda)$ is nonzero,
	\item $d(\chi_i) = 0$ for all $i$, and
	\item $\Top(\lambda) = p - 1$ and $\wt(\lambda) \geq q$;
	\end{itemize}
	making $\lambda$ the only basis element with a nonzero image.
\end{lemma}

\begin{proof}
Partition a basis for $\Gamma$, reordering if necessary, as
\[
\{\gamma_0, \dots, \gamma_r \} \cup \{ \gamma_{r+1}, \dots, \gamma_{r+s} \} \cup \{\gamma_{r+s+1}, \dots, \gamma_{r+s+t} \},
\]
where the basis elements $\gamma_0, \dots, \gamma_r$ have nonzero images in the top cone $\Mt^+\langle \nu \rangle$, basis elements $\gamma_{r+1}, \dots, \gamma_{r+s}$ have nonzero images in the bottom cone $\Mt^-\langle \nu \rangle$, and $d$ is zero on $\gamma_{r+s+1}, \dots, \gamma_{r+s+t}$.  We allow for the possibility that $s=0$ if there are no basis elements supporting nonzero maps to the bottom cone, or $t=0$ if there are no basis elements mapped to zero. Set
\begin{align*}
	\Delta^+ &= \Mtb{\gamma_0, \dots, \gamma_r}\\
	\Delta^- &=  \Mtb{\gamma_{r+1}, \dots, \gamma_{r+s}}\\
	\Delta^0 &= \Mtb{\gamma_{r+s+1}, \dots, \gamma_{r+s+t}}
\end{align*}
so that $\Gamma \cong \Delta^+ \oplus \Delta^- \oplus \Delta^0$.

We will find a new basis $\{\lambda, \chi_1, \dots, \chi_r\}$ for $\Delta^+$ satisfying $d(\lambda) \neq 0$ and $d(\chi_i) = 0$ for each $i$.  For degree reasons, any $\gamma_i$ with nonzero image in the top cone of $\nu$ has $\Top(\gamma_i) = p-1$.
So for $0 \leq i \leq r$, the image $d(\gamma_i) = \tau^{k_i}\nu$ for $k_i \geq 0$.  Reindexing if necessary, we can assume $\gamma_0$ is of minimal weight among $\gamma_0, \dots, \gamma_r$.  Rename it $\lambda := \gamma_0$ so that $d(\lambda) = \tau^{k_0}\nu$ and $k_i \geq k_0$ for all $i$.
Since $\lambda$ has the lowest weight, each remaining $\gamma_i$ lies in the upper cone of $\lambda$ as in Figure \ref{basis upper}.

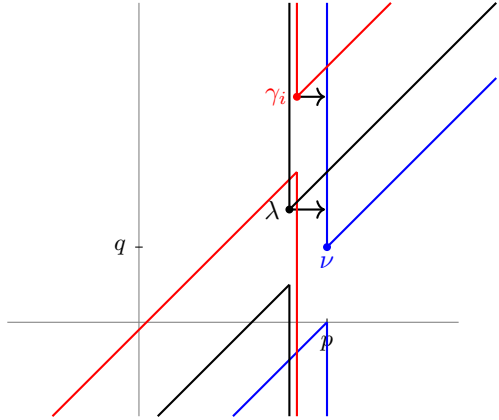
\begin{figure}[ht]
\begin{center}
\begin{tikzpicture}[scale=0.5]
\draw[gray] (-3.5,0) -- (8.5,0);
\draw[gray] (0,-2.5) -- (0,8.5);

\draw (5,0.1) -- (5,-0.1) node[below] {\small $p$};
\draw (0.1,2) -- (-0.1,2) node[left] {\small $q$};

\draw[->,thick] (4,3) -- (4.95,3);
\draw[->,thick] (4.2,6) -- (4.95,6);

\draw[thick, blue] (5,2) -- (9.5,6.5);
\draw[thick, blue] (5,2) -- (5,8.5);
\draw[thick, blue] (5,0) -- (5,-2.5);
\draw[thick, blue] (5,0) -- (2.5,-2.5);
\draw[blue] (5,2) node[below] {$\nu$};
\fill[blue] (5,2) circle(3pt);

\draw[thick, red] (4.2,6) -- (6.7,8.5);
\draw[thick, red] (4.2,6) -- (4.2,8.5);
\draw[thick, red] (4.2,4) -- (4.2,-2.5);
\draw[thick, red] (4.2,4) -- (-2.3,-2.5);
\draw[red] (4.2,6) node[left] {$\gamma_i$};
\fill[red] (4.2,6) circle(3pt);

\draw[thick] (4,3) -- (9.5,8.5);
\draw[thick] (4,3) -- (4,8.5);
\draw[thick] (4,1) -- (4,-2.5);
\draw[thick] (4,1) -- (0.5,-2.5);
\draw (4,3) node[left] {$\lambda$};
\fill (4,3) circle(3pt);

\end{tikzpicture}
\end{center}
\caption{Nonzero images in the upper cone.}\label{basis upper}
\end{figure}

Next we will use $\lambda$ to form the $\chi_i$.  For each $i$ with $1 \leq i \leq r$, replace $\gamma_i$ with $\chi_i := \gamma_i + \tau^{k_i - k_0}\lambda$.  Notice that $\chi_i$ is a homogeneous element in the same bidegree as $\gamma_i$ and that $\chi_i \in \ker(d)$.  The $\chi_i$ are all independent, as a dependence among them would give rise to a dependence among the $\gamma_i$.
Indeed, suppose there existed nonzero elements $M_i\in\Mt$ for $1\le i\le r$ such that each $M_i\chi_i$ is in the same bidegree and $\sum M_i\chi_i=0$. Then
\[
\sum_{i=1}^rM_i\gamma_i+\sum_{i=1}^r M_i\tau^{k_i-k_0}\lambda=\sum_{i=1}^r M_i\gamma_i+\left(\sum_{i=1}^r M_i\tau^{k_i-k_0}\right)\gamma_0=0.
\]
The coefficient of $\gamma_0$ could be zero, but there is at least one $\gamma_i$ with $i > 0$ that has nonzero coefficient $M_i$.  Thus we have a dependence among the basis elements $\gamma_i$, a contradiction.\footnote{Not every $\Mt$-combination of the $\gamma_i$ is independent.
For example, $\{\gamma_1, \gamma_1 + \theta\gamma_2 \}$ is dependent because $\tau \gamma_1 = \tau (\gamma_1 + \theta\gamma_2)$.}  We can therefore include a free module and quotient to get the short exact sequence
\[0\to \bigoplus_{i=1}^r \Mt\langle \chi_i\rangle \to \Delta^+\to Q\to 0,\]
where $Q\iso\Mt\langle\lambda\rangle$.  So $\Delta^+$ is isomorphic to $\Mt\langle \lambda,\chi_1,\dots,\chi_r \rangle$ since $\Mt$ is self-injective. Now we have a basis for $\Delta^+$ with $d(\lambda) \neq 0$ but $d(\chi_i) = 0$ for each $i$.

We use a similar approach to make a change of basis for $\Delta^- =  \Mtb{\gamma_{r+1}, \dots, \gamma_{r+s}}$, now using $\lambda$ to modify the basis elements with nonzero images in the bottom cone.
For $r+1 \leq i \leq r+s$, we have $d(\gamma_i) = \frac{\theta}{\rho^{j_i} \tau^{k_i}} \nu$ for some $j_i, k_i \geq 0$ and it must be the case that $\Top(\gamma_i) \leq p-1$ and $\wt(\gamma_i) \leq q - 2$.
In particular, for $r+1 \leq i \leq r+s$, each basis element $\gamma_i$ has a bidegree that lies inside the lower cone of $\lambda$ as in Figure \ref{basis mixed}.

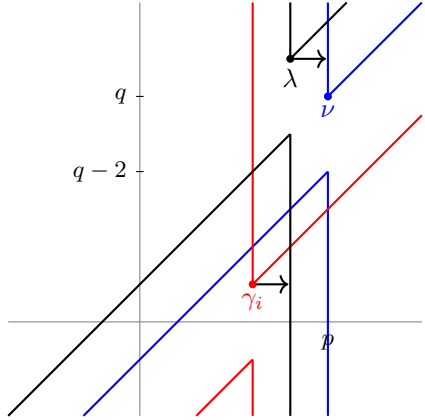
\begin{figure}[ht]
\begin{center}
\begin{tikzpicture}[scale=0.5]
\draw[gray] (-3.5,0) -- (7.5,0);
\draw[gray] (0,-2.5) -- (0,8.5);

\draw (5,0.1) -- (5,-0.1) node[below] {\small $p$};
\draw (0.1,6) -- (-0.1,6) node[left] {\small $q$};
\draw (0.1,4) -- (-0.1,4) node[left] {\small $q-2$};

\draw[->,thick] (4,7) -- (4.95,7);
\draw[->,thick] (3,1) -- (3.95,1);

\draw[thick, blue] (5,6) -- (7.5,8.5);
\draw[thick, blue] (5,6) -- (5,8.5);
\draw[thick, blue] (5,4) -- (5,-2.5);
\draw[thick, blue] (5,4) -- (-1.5,-2.5);
\draw[blue] (5,6) node[below] {$\nu$};
\fill[blue] (5,6) circle(3pt);

\draw[thick, red] (3,1) -- (7.5,5.5);
\draw[thick, red] (3,1) -- (3,8.5);
\draw[thick, red] (3,-1) -- (3,-2.5);
\draw[thick, red] (3,-1) -- (1.5,-2.5);
\draw[red] (3,1) node[below] {$\gamma_i$};
\fill[red] (3,1) circle(3pt);

\draw[thick] (4,7) -- (5.5,8.5);
\draw[thick] (4,7) -- (4,8.5);
\draw[thick] (4,5) -- (4,-2.5);
\draw[thick] (4,5) -- (-3.5,-2.5);
\draw (4,7) node[below] {$\lambda$};
\fill (4,7) circle(3pt);

\end{tikzpicture}
\end{center}
\caption{Images from $\Delta^-$ and $\Delta^+$.}\label{basis mixed}
\end{figure}

Recall that we defined $\lambda := \gamma_0$ above so $d(\lambda) = \tau^{k_0}\nu$.  Replace the basis element $\gamma_i$ (for $r+1 \leq i \leq r+s$) with $\chi_i := \gamma_i + \frac{\theta}{\rho^{j_i} \tau^{k_i + k_0}}\lambda$.  Again the coefficient of $\lambda$ has been chosen so that $\chi_i$ is a homogenous element generating an $\Mt$ in the same bidegree as $\gamma_i$ and $d(\chi_i) = 0$.
As before, the $\chi_i$ are independent.  Thus we have a new basis for $\Delta^-$ given by $\chi_{r+1}, \dots, \chi_{r+s}$.

Finally, for $\Delta^0 = \Mtb{\gamma_{r+s+1}, \dots, \gamma_{r+s+t}}$ no real change of basis is required.  The map $d$ is already zero on each basis element, so simply define $\chi_i := \gamma_i$ for $r+s+1 \leq i \leq r+s+t$.  Taking the union of the new bases, we now have a basis for $\Gamma \cong \Delta^+ \oplus \Delta^- \oplus \Delta^0$ of the form $\lambda, \chi_1, \dots, \chi_m$ with only $\lambda$ having a nonzero image.
\end{proof}

In the next case, a change of basis for the free module $\Gamma$ will produce a particularly nice subbasis of free generators called a ``ramp,'' as described in \cite{K} (also referred to as a ``stairstep" pattern in \cite{FL}).

\begin{defn}\label{rampdef}
	A collection $\omega_1, \dots, \omega_n$ of homogeneous elements satisfies the \mdfn{ramp condition} if $\Top(\omega_i) < \Top(\omega_{i+1})$ and $\fix(\omega_i) < \fix(\omega_{i+1})$ for each $i$.
	Such elements are referred to as a \mdfn{ramp} of length $n$.
\end{defn}

Note that the increasing topological dimension and increasing fixed-set dimension in the ramp condition means each $\omega_{i+1}$ is to the right of $\omega_i$ and on a lower diagonal.  If $\omega_1, \dots, \omega_n$ are generators of a free $\Mt$-module forming a ramp, no $\omega_i$ lies in a bidegree that sits inside the upper cone $\Mt^+\langle \omega_j \rangle$ for $i \neq j$.
An example of $\Mt$ generators satisfying the ramp condition is depicted in Figure \ref{ramp figure}.  The details of the ramp condition will not play a big role in the proof of the freeness theorem, however they will be crucial in Section \ref{shifts} where we determine the bidegrees of the free generators after a nontrivial differential.

\begin{figure}[ht]
\begin{center}
\begin{tikzpicture}[scale=0.4]
\draw (0,4) -- (2.5,6.5);
\draw (0,4) -- (0,6.5);
\draw (0,2) -- (0,-3.5);
\draw (0,2) -- (-5.5,-3.5);
\fill (0,4) circle (3.75pt) node[below] {$\omega_1$};

\draw (2,2) -- (6.5,6.5);
\draw (2,2) -- (2,6.5);
\draw (2,0) -- (2,-3.5);
\draw (2,0) -- (-1.5,-3.5);
\fill (2,2) circle (3.75pt) node[below] {$\omega_2$};

\draw (3,2) -- (7.5,6.5);
\draw (3,2) -- (3,6.5);
\draw (3,0) -- (3,-3.5);
\draw (3,0) -- (-0.5,-3.5);
\fill (3,2) circle (3.75pt) node[below] {$\omega_3$};

\draw (4,0) -- (10.5,6.5);
\draw (4,0) -- (4,6.5);
\draw (4,-2) -- (4,-3.5);
\draw (4,-2) -- (2.5,-3.5);
\fill (4,0) circle (3.75pt) node[below] {$\omega_4$};

\draw (7,2) -- (11.5,6.5);
\draw (7,2) -- (7,6.5);
\draw (7,0) -- (7,-3.5);
\draw (7,0) -- (3.5,-3.5);
\fill (7,2) circle (3.75pt) node[below] {$\omega_5$};
\end{tikzpicture}
\end{center}
\caption{A ramp of length 5.}\label{ramp figure}
\end{figure}
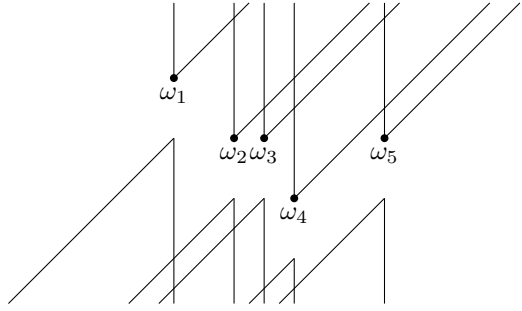

Now we prove a change of basis for a free module $\Gamma$ supporting a nonzero map only to the bottom cone of $\Sigma^{p,q}\Mt = \Mtb{\nu}$.  After the change of basis, there will be a ramp of generators supporting nonzero maps to the bottom cone.  An example is depicted in Figure \ref{bottom cone figure}.  Again, the image of $d$ is determined on the generators, so only these arrows are shown in Figure \ref{bottom cone figure}.  However, $d$ is nonzero in (finitely many) other bidegrees.
For example, if $d(\omega_1) = \frac{\theta}{\rho^2}\nu$ then also $d(\rho\, \omega_1) = \frac{\theta}{\rho}\nu$ and $d(\rho^2 \omega_1) = \theta\nu$.  As any product of elements in $\Mt^-$ is zero, $d$ is zero on $\Mt^-\langle \omega_1 \rangle$.

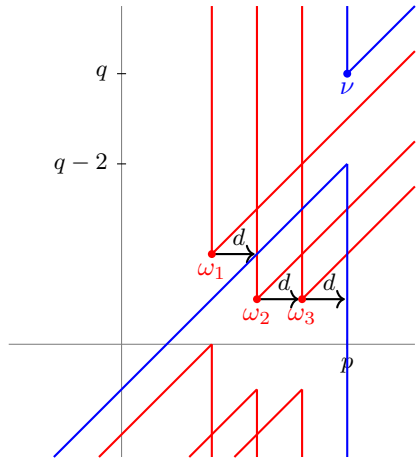
\begin{figure}[ht]
\begin{center}
\begin{tikzpicture}[scale=0.6]
\draw[gray] (-2.5,0) -- (6.5,0);
\draw[gray] (0,-2.5) -- (0,7.5);

\draw (5,0.1) -- (5,-0.1) node[below] {\small $p$};
\draw (0.1,6) -- (-0.1,6) node[left] {\small $q$};
\draw (0.1,4) -- (-0.1,4) node[left] {\small $q-2$};

\draw[->,thick] (2,2) -- (2.95,2);
\draw (2.6,2) node[above] {\small $d$};
\draw[->,thick] (3,1) -- (3.93,1);
\draw (3.6,1) node[above] {\small $d$};
\draw[->,thick] (4,1) -- (4.95,1);
\draw (4.6,1) node[above] {\small $d$};

\draw[thick, red] (2,2) -- (6.5,6.5);
\draw[thick, red] (2,2) -- (2,7.5);
\draw[thick, red] (2,0) -- (2,-2.5);
\draw[thick, red] (2,0) -- (-0.5,-2.5);
\draw[red] (2,2) node[below] {$\omega_1$};
\fill[red] (2,2) circle(2.5pt);

\draw[thick, red] (3,1) -- (6.5,4.5);
\draw[thick, red] (3,1) -- (3,7.5);
\draw[thick, red] (3,-1) -- (3,-2.5);
\draw[thick, red] (3,-1) -- (1.5,-2.5);
\draw[red] (3,1) node[below] {$\omega_2$};
\fill[red] (3,1) circle(2.5pt);

\draw[thick, red] (4,1) -- (6.5,3.5);
\draw[thick, red] (4,1) -- (4,7.5);
\draw[thick, red] (4,-1) -- (4,-2.5);
\draw[thick, red] (4,-1) -- (2.5,-2.5);
\draw[red] (4,1) node[below] {$\omega_3$};
\fill[red] (4,1) circle(2.5pt);

\draw[thick, blue] (5,6) -- (6.5,7.5);
\draw[thick, blue] (5,6) -- (5,7.5);
\draw[thick, blue] (5,4) -- (5,-2.5);
\draw[thick, blue] (5,4) -- (-1.5,-2.5);
\draw[blue] (5,6) node[below] {$\nu$};
\fill[blue] (5,6) circle(2.5pt);

\end{tikzpicture}\hfill
\end{center}
\caption{A ramp of generators mapping to the bottom cone of $\Mtb{\nu}$.}
\label{bottom cone figure}
\end{figure}

\begin{lemma}\label{bottom cone lemma}
Consider a nonzero graded $\Mt$-module homomorphism with bidegree $(1,0)$ of the form $\Gamma = \Mtb{ \gamma_0, \gamma_1, \dots, \gamma_{m}} \xrightarrow{d} \Mtb{\nu}$ where $|\nu| = (p,q)$.  Assume for all $i$ that $\Top(\gamma_i) \leq p$ and whenever $\Top(\gamma_i) = p$ that $\wt(\gamma_i) \leq q$.

Suppose $\im d$ takes nonzero values only in $\Mt^-\langle \nu \rangle$. Then there is a change of basis for $\Gamma$ so that
\[
\Gamma \cong \Mtb{ \omega_1, \dots, \omega_n, \chi_{n+1}, \dots, \chi_{m+1} },
\]
where
	\begin{itemize}
	\item $d(\omega_i)$ is nonzero for all $i$,
	\item $d(\chi_i) = 0$ for all $i$,
	\item $\wt(\omega_i) \leq q-2$ for all $i$, and
	\item $\omega_1, \dots, \omega_n$ forms a ramp.
	\end{itemize}
	That is, $\omega_1, \dots, \omega_n$ are the only basis elements with nonzero images and they have increasing topological dimension and increasing fixed-set dimension.
\end{lemma}

\begin{proof}
As in the proof of Lemma \ref{top cone lemma}, we partition a basis for $\Gamma$, reordering if necessary.  Now, by assumption, there are no nonzero images in $\Mt^+\langle \nu \rangle$.  Partition $\Gamma$ as
\[
\{\gamma_0, \dots, \gamma_s \} \cup \{\gamma_{s+1}, \dots, \gamma_{s+t} \},
\]
where the basis elements $\gamma_{0}, \dots, \gamma_{s}$ have nonzero images in the bottom cone $\Mt^-\langle \nu \rangle$, and $d$ is zero on $\gamma_{s+1}, \dots, \gamma_{s+t}$.
We allow for the possibility that $t=0$ if there are no basis elements mapped to zero. Set
\begin{align*}
	\Delta^- &=  \Mtb{\gamma_{0}, \dots, \gamma_{s}}\\
	\Delta^0 &= \Mtb{\gamma_{s+1}, \dots, \gamma_{s+t}}
\end{align*}
so that $\Gamma \cong \Delta^- \oplus \Delta^0$.

We will find a new basis $\omega_1, \dots, \omega_n$ forming a ramp and supporting nonzero maps to the bottom cone.  Consider any two distinct basis elements $\gamma_a$ and $\gamma_b$, $0 \leq a,b \leq s$.  Since $\gamma_a$ and $\gamma_b$ support nonzero images in the bottom cone of $\nu$, we can write $d(\gamma_a) = \frac{\theta}{\rho^{j_a} \tau^{k_a}} \nu$ and $d(\gamma_b) = \frac{\theta}{\rho^{j_b} \tau^{k_b}} \nu$.
If $\gamma_a$ and $\gamma_b$ do not satisfy the ramp condition, then one lies inside the range of the upper cone of the other, as shown in Figure \ref{basis lower}.  Without loss of generality, we can assume $\gamma_b$ lies inside the upper cone of $\gamma_a$ so that $\Top(\gamma_a) \leq \Top(\gamma_b)$ and $\fix(\gamma_a) \geq \fix(\gamma_b)$.

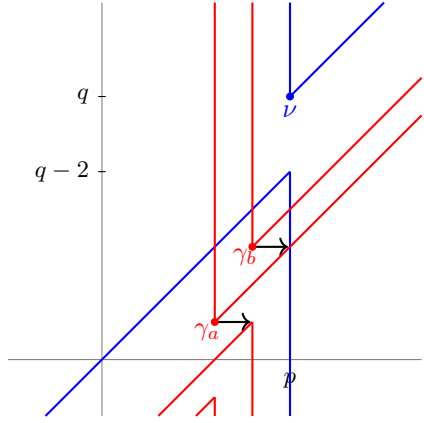
\begin{figure}[ht]
\begin{center}
\begin{tikzpicture}[scale=0.5]
\draw[gray] (-2.5,-1) -- (8.5,-1);
\draw[gray] (0,-2.5) -- (0,8.5);

\draw (5,-0.9) -- (5,-1.1) node[below] {\small $p$};
\draw (0.1,6) -- (-0.1,6) node[left] {\small $q$};
\draw (0.1,4) -- (-0.1,4) node[left] {\small $q-2$};

\draw[->,thick] (4,2) -- (4.95,2);
\draw[->,thick] (3,0) -- (3.95,0);

\draw[thick, blue] (5,6) -- (7.5,8.5);
\draw[thick, blue] (5,6) -- (5,8.5);
\draw[thick, blue] (5,4) -- (5,-2.5);
\draw[thick, blue] (5,4) -- (-1.5,-2.5);
\draw[blue] (5,6) node[below] {$\nu$};
\fill[blue] (5,6) circle(3pt);

\draw[thick, red] (3,0) -- (8.5,5.5);
\draw[thick, red] (3,0) -- (3,8.5);
\draw[thick, red] (3,-2) -- (3,-2.5);
\draw[thick, red] (3,-2) -- (2.5,-2.5);
\draw[red] (2.8,0.2) node[below] {$\gamma_a$};
\fill[red] (3,0) circle(3pt);

\draw[thick, red] (4,2) -- (8.5,6.5);
\draw[thick, red] (4,2) -- (4,8.5);
\draw[thick, red] (4,0) -- (4,-2.5);
\draw[thick, red] (4,0) -- (1.5,-2.5);
\draw[red] (3.8,2.2) node[below] {$\gamma_b$};
\fill[red] (4,2) circle(3pt);

\end{tikzpicture}
\end{center}
\caption{Nonzero images in the lower cone.}\label{basis lower}
\end{figure}

Multiplication by $\rho$ preserves fixed-set dimension but increases topological dimension by one.  Multiplication by $\tau$ preserves topological dimension but decreases fixed-set dimension by one.  Thus $j_a \geq j_b$ and $k_a \geq k_b$.  In the basis for $\Delta^-$ replace $\gamma_b$ with $\chi_b := \rho^{j_a - j_b} \tau^{k_a - k_b} \gamma_a + \gamma_b$.  Once more $\chi_b$ is defined to be homogeneous in the same bidegree as $\gamma_b$ and $d(\chi_b) = 0$.
As in the proof of Lemma \ref{top cone lemma}, the independence of $\gamma_a$ and $\gamma_b$ will imply that $\gamma_a$ and $\chi_b$ are independent as well.
So $\Mtb{\gamma_a,\gamma_b} \cong \Mtb{\gamma_a, \chi_b}$.

Continue to reduce the set of $\gamma_i$ supporting nonzero images in the bottom cone in this way until no $\gamma_b$ lies in the upper cone region of any $\gamma_a$.  The remaining $\gamma_i$ with $d(\gamma_i) \neq 0$ satisfy the ramp condition up to reindexing.  Let $n$ be the number of such $\gamma_i$ with nonzero images remaining, reindex these to have increasing topological dimension and rename them $\omega_i := \gamma_i$ for $1\le i\le n$.  Then reindexing the $\chi_i$ as necessary, we now have a basis for $\Delta^-$ of the form $\omega_1, \dots, \omega_n, \chi_{n+1}, \dots, \chi_{s}$.

Finally for $\Delta^0 = \Mtb{\gamma_{s+1}, \dots, \gamma_{s+t}}$, no real change of basis is required as $d$ is already zero on each basis element. Define $\chi_i := \gamma_i$ for $s+1 \leq i \leq s+t$.  Taking the union of the new bases, we have a basis for $\Gamma \cong \Delta^- \oplus \Delta^0$ of the form $\omega_1, \dots, \omega_n, \chi_{n+1}, \dots, \chi_{m+1}$ as desired.
\end{proof}

\section{Change of basis for attaching a representation cell}\label{basis for attaching}

In the previous section, we introduced two algebraic changes of basis motivated by Kronholm's argument.  In this section we will deduce some consequences for the cohomology of a $\Rep(C_2)$-complex as an $\Mt$-module, completing the proof of Lemma 3.1 from \cite{K}.

Here we show that when attaching a single cell to a $\Rep(C_2)$-complex, up to a change of basis, a nonzero differential in the long exact sequence takes one of only two forms.  The first looks much like the result of Lemma \ref{top cone lemma}, where a single basis element $\lambda$ supports a nonzero differential to the top cone.  In the topological setting, we obtain a further restriction on the weight of $\lambda$ using $\rho$-localization.  The second form follows directly from Lemma \ref{bottom cone lemma}, where basis elements satisfying the ramp condition (see Definition \ref{rampdef}) support nonzero differentials to the bottom cone.

To begin, recall the following lemma and subsequent remark that appear in \cite[Lemma 4.3 and Remark 4.4]{CM} relating the $\rho$-localization of equivariant cohomology to the singular cohomology of the fixed set.

\begin{lemma}\label{rho localization}
($\rho$-localization) Let $X$ be a finite $C_2$-CW complex.  Then
\[
\rho^{-1}H^{*,*}(X) \cong \rho^{-1}H^{*,*}(X^{C_2}) \cong H^{*}_{\sing}(X^{C_2}) \otimes_{\F_2} \rho^{-1}\Mt.
\]
\end{lemma}

\begin{remark}\label{tau-torsion remark}
An important consequence of the previous lemma is that $\rho^{-1}\tilde{H}^{*,*}(X)$ does not have any $\tau$-torsion, since $\rho^{-1} \Mt \cong \F_2[\tau,\rho^\pm]$ and $\rho^{-1}\tilde{H}^{*,*}(X)$ is free as a $\rho^{-1}\Mt$-module.
\end{remark}

We will use this fact to show the long exact sequence for attaching a single cell to a $\Rep(C_2)$-complex cannot have any non-surjective differentials into the top cone.  That is, if there is a nontrivial differential into the top cone as in Lemma \ref{top cone lemma}, so that $d(\lambda)=\tau^k \nu$ for some $k$, then in fact $\wt(\lambda) = \wt(\nu)$, $k=0$, and the restriction $\Mtb{\lambda} \to \Mtb{\nu}$ is an isomorphism.

The following result is quite similar to \cite[Lemma 3.1]{K} although the proof differs slightly.  Our proof uses Lemmas \ref{top cone lemma} and \ref{bottom cone lemma}, together with Remark \ref{tau-torsion remark} to further restrict differentials to the top cone. As in Section \ref{change of basis}, the hypotheses involving topological dimensions and weights are motivated by the inductive step of the main theorem.
Recall from the proofs of Lemmas \ref{top cone lemma} and \ref{bottom cone lemma}, every change of basis we made preserved the bidegrees of the free generators.  As these are the only change of bases used to prove the following result, the topological dimensions and weights of basis elements are preserved.

\begin{lemma} \label{basis update}
Let $B$ be a $\Rep(C_2)$-complex with reduced cohomology given by a graded free $\Mt$-module with basis $\gamma_0, \dots, \gamma_m$ so $\tilde{H}^{*,*}(B) \cong \Mtb{\gamma_0, \dots, \gamma_m}$.  Suppose $X$ is obtained from $B$ by attaching a single $(p,q)$ cell and let $\nu$ denote the generator for the reduced cohomology of $X/B \cong S^{p,q}$.  Assume further that for all $i$, $\Top(\gamma_i) \leq p$ and whenever $\Top(\gamma_i) = p$ then $\wt(\gamma_i) \leq q$.

The cofiber sequence
\[
B \xrightarrow{\iota} X \xrightarrow{\pi} S^{p,q}
\]
gives rise to the differential $d : \tilde{H}^{*,*}(B) \to \tilde{H}^{*+1,*} (S^{p,q})$ in the long exact sequence.  After an appropriate change of basis, one of the following is true:
\begin{enumerate}
\item[$(a)$] the differential $d \equiv 0$;
\item[$(b)$] the differential is zero on every basis element except $\lambda$, where $|\lambda| = (p-1,q)$ and the restriction of the differential to $\Mtb{\lambda} \to \Mtb{\nu}$ is an isomorphism; or
\item[$(c)$] the basis elements supporting nonzero differentials $\omega_1, \dots, \omega_n$ satisfy the ramp condition and map to $\Mt^-\langle \nu \rangle$.
\end{enumerate}
\end{lemma}

\begin{proof}

If $d \equiv 0$, we are done.  Otherwise, $d$ either has nonzero image in $\Mt^+\langle\nu\rangle$ or not.  If $d$ does have nonzero image in $\Mt^+\langle\nu\rangle$, then Lemma \ref{top cone lemma} shows we can rewrite $\tilde{H}^{*,*}(B) \cong \Mtb{\lambda, \chi_1, \dots, \chi_{m}}$ where $\Top(\lambda) = p-1$, $\wt(\lambda) \geq q$, and $d(\chi_i) = 0$ for all $i$.
If $\wt(\lambda) = q$, then the restriction of $d$ to $\Mtb{\lambda} \to \Mtb{\nu}$ is an isomorphism.  For case $(b)$, it remains to show this is the only possible weight.  Suppose instead that $\wt(\lambda) > q$ so that $d(\lambda) = \tau^k \nu$ for some $k > 0$.  Then the following equation holds.\footnote{Here we deviate from Kronholm, as there is a small error in the calculation of $\cok(d)$ and $\ker(d)$ in \cite{K}.}
\[
d\left(\frac{\theta}{\tau^k}\lambda \right) = \frac{\theta}{\tau^k}d(\lambda) = \frac{\theta}{\tau^k}\tau^k \nu = \theta \nu
\]
A portion of $\cok(d)$ and $\ker(d$) are depicted in Figure \ref{diff upper}.  (All of the $\chi_i$ are also in $\ker(d)$ but are not depicted here.)

\begin{figure}[ht]
\begin{center}
\begin{tikzpicture}[scale=0.6]
\draw[gray] (-3.5,0) -- (5.5,0);
\draw[gray] (0,-2.5) -- (0,6.5);

\draw (2,0.1) -- (2,-0.1) node[below] {\small $p$};
\draw (0.1,2) -- (-0.1,2) node[left] {\small $q$};

\draw[->,thick] (1,4) -- (1.95,4);
\draw (1.6,4) node[above] {\small $d$};

\draw[thick, red] (1,4) -- (3.5,6.5);
\draw[thick, red] (1,4) -- (1,6.5);
\draw[thick, red] (1,2) -- (1,-2.5);
\draw[thick, red] (1,2) -- (-3.5,-2.5);
\draw[red] (1,4) node[below] {$\lambda$};
\fill[red] (1,4) circle(2.5pt);

\draw[thick, blue] (2,2) -- (5.5,5.5);
\draw[thick, blue] (2,2) -- (2,6.5);
\draw[thick, blue] (2,0) -- (2,-2.5);
\draw[thick, blue] (2,0) -- (-0.5,-2.5);
\draw[blue] (2,2) node[below] {$\nu$};
\fill[blue] (2,2) circle(2.5pt);

\end{tikzpicture} \hfill
\begin{tikzpicture}[scale=0.6]
\draw[gray] (-3.5,0) -- (5.5,0);
\draw[gray] (0,-2.5) -- (0,6.5);

\draw (2,0.1) -- (2,-0.1) node[below] {\small $p$};
\draw (0.1,2) -- (-0.1,2) node[left] {\small $q$};

\draw[thick, red] (1,2) -- (1,1) -- (-2.5,-2.5);
\draw[thick, red] (1,2) -- (-3.5,-2.5);
\draw[red] (-1.25,1) node {\small $\ker(d)$};

\draw[thick, blue] (2,2) -- (5.5,5.5);
\draw[thick, blue] (2,2) -- (2,3) -- (5.5,6.5);
\draw[blue] (4,3) node {\small $\cok(d)$};
\end{tikzpicture}\hfill
\end{center}
\caption{Differential to the upper cone. Note $d$ surjects onto $\Mt^-\langle \nu\rangle$.}\label{diff upper}
\end{figure}
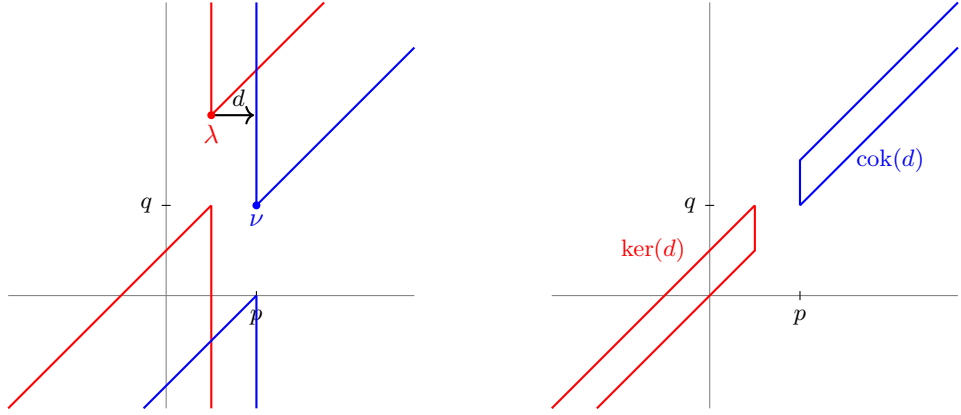

To actually compute $\tilde{H}^{*,*}(X)$ would require solving an extension problem of graded $\Mt$-modules as $\tilde{H}^{*,*}(X)$ is in the middle of the short exact sequence
\[
0 \to \cok(d) \to \tilde{H}^{*,*}(X) \to \ker(d)  \to 0.
\]
However, we do not need to solve this extension problem because the $\rho$-localization of $\tilde{H}^{*,*}(X)$ is actually independent of the resolution.  By Remark \ref{tau-torsion remark}, we know $\rho^{-1}\tilde{H}^{*,*}(X)$ cannot have any $\tau$-torsion.
Since $\tau^{k-1}\nu \notin \im(d)$, by exactness we have $\tau^{k-1}\nu \notin \ker(\pi^*)$ and hence a nonzero class $\pi^*(\tau^{k-1}\nu) \in \tilde{H}^{*,*}(X)$.
This class has no $\rho$-torsion and thus survives $\rho$-localization.\footnote{Without the assumption that for all $i$ $\Top(\gamma_i) \leq p$ and if $\Top(\gamma_i) = p$ then $\wt(\gamma_i) \leq q$, it is possible to have differentials that introduce $\rho$-torsion in $\tilde{H}^{*,*}(X)$.
See Example 6.6 in \cite{CM}.}  Yet it does have $\tau$-torsion since $\tau \cdot \pi^*(\tau^{k-1}\nu) = \pi^*(\tau^k \nu) = 0$,
contradicting Lemma \ref{rho localization}.  Thus the basis given by Lemma \ref{top cone lemma} reduces to case $(b)$.

Finally, the remaining possibility is that $d$ is nonzero but does not have nonzero image in $\Mt^+\langle\nu\rangle$.  Then the assumptions of Lemma \ref{bottom cone lemma} are met and, after a change of basis for $\tilde{H}^{*,*}(B)$, a ramp of basis elements $\omega_1, \dots, \omega_n$ support nonzero differentials to $\Mt^-\langle\nu\rangle$, as stated in case $(c)$.
\end{proof}

\section{Freeness theorem}\label{main section}

We are now ready to prove Kronholm's freeness theorem for finite $\Rep(C_2)$-complexes.  The proof will proceed by induction on the number of representation cells and consider the attaching map for a single cell.  The bulk of the proof will involve case $(c)$ of Lemma \ref{basis update}, with a ramp (see Definition \ref{rampdef}) of $\Mt$ generators supporting differentials to the lower cone of another generator.  This will lead to the cohomology being a free module with the same number of $\Mt$ generators, but in shifted bidegrees from their original positions.

In this main case, we use $\tau$-localization to show that the cohomology is free.  However, we will need to show the inductive hypothesis holds, namely that any generators in the highest topological dimension are below a certain weight.  For that, we use $\tau$-localization together with $\rho$-localization.  Despite the generators appearing in new bidegrees, the set of topological dimensions of the generators as well as the set of fixed-set dimensions are preserved.  This will give us the constraint on the weights required to complete the inductive step.

Somewhat surprisingly, in this main case we prove the cohomology is free without ever identifying the free generators or the shifts.  For the purpose of computations, a choice of free basis and a precise formula for calculating the Kronholm shifts are given in Section \ref{shifts}.

\begin{thm} \label{main thm}
(Freeness theorem) If $X$ is a finite $\Rep(C_2)$-complex then $\tilde{H}^{*,*}(X;\underline{\F_2})$ is free as a graded $\Mt$-module, where $\Mt = H^{*,*}(\pt;\underline{\F_2})$.
\end{thm}

\begin{proof}
Attaching one cell at a time, we can filter $X$ so that each $X_{k+1}$ is formed from $X_k$ by attaching a single $\Rep(C_2)$-cell, $e_{k+1}$.  Order the representation cells $e_1, e_2, \dots, e_{K}$ by increasing topological dimension and increasing weight so that if $e_i \cong D(\R^{p_i, q_i})$ then $p_i \leq p_{i+1}$ for all $i$, and if $p_i = p_{i+1}$ then $q_i \leq q_{i+1}$.  We proceed by induction on the spaces in the `one-at-a-time' cellular filtration\footnote{Although we filter $X$ using the `one-at-a-time' cellular filtration as in Kronholm's proof \cite{K}, we prove the theorem using the long exact sequence associated to a cofiber sequence with no reference to the cellular spectral sequence.  This is addressed in more detail in Section \ref{Kronholm proof}.}
\[
\pt = X_0 \subseteq X_1 \subseteq \cdots \subseteq X_k \subseteq X_{k+1} \subseteq \cdots \subseteq X_K = X.
\]
The base case is trivial.

We will inductively prove that each $\tilde{H}^{*,*}(X_k)$ is a free $\Mt$-module, where all the generators have topological dimensions between zero and $p_k$ and any generators in topological dimension $p_k$ have weights between zero and $q_k$.
For the inductive step we will form $X_{k+1}$ from $X_k$ by attaching a single $(p_{k+1},q_{k+1})$-cell.  In the course of the inductive step we will see that the generator of $\tilde{H}^{*,*}(X_{k+1})$ corresponding to this new cell may remain in bidegree $(p_{k+1},q_{k+1})$, ``shift" to a lower weight, or vanish.

Now assume the inductive hypothesis holds for $\tilde{H}^{*,*}(X_k)$.  To simplify notation we set $p = p_{k+1}$ and $q = q_{k+1}$.  We will show $\tilde{H}^{*,*}(X_{k+1})$ is free with all generators in topological dimension $p$ having weight at most $q$.

The cofiber sequence
\[
X_k \xrightarrow{\iota} X_{k+1} \xrightarrow{\pi} X_{k+1}/X_k \cong S^{p,q}
\]
induces the long exact sequence
\[
\cdots \xrightarrow{d} \tilde{H}^{*,*}(S^{p,q}) \xrightarrow{\pi^*} \tilde{H}^{*,*}(X_{k+1}) \xrightarrow{\iota^*} \tilde{H}^{*,*}(X_k) \xrightarrow{d} \tilde{H}^{*+1,*}(S^{p,q}) \xrightarrow{\pi^*} \cdots.
\]
Let $\nu$ denote the free generator of $\tilde{H}^{*,*} (X_{k+1}/X_k) \cong \tilde{H}^{*,*}(S^{p,q}) \cong \Sigma^{p,q} \Mt$ in bidegree $(p,q)$.  To show that $\tilde{H}^{*,*} (X_{k+1})$ is a free $\Mt$-module we will solve the extension problem of graded $\Mt$-modules that appears in the short exact sequence
\[
0 \to \cok(d) \to \tilde{H}^{*,*}(X_{k+1}) \to \ker(d) \to 0.
\]
We proceed by investigating the differential $d: \tilde{H}^{*,*}(X_k) \to \tilde{H}^{*+1,*}(S^{p,q})$.

Lemma \ref{basis update} enumerates the possibilities. In case (a), the differential $d \equiv 0$ and the extension problem is easily solved.  The short exact sequence becomes
\[
0 \to \Mtb{\nu} \to \tilde{H}^{*,*}(X_{k+1}) \to \tilde{H}^{*,*}(X_k) \to 0,
\]
which splits since $\tilde{H}^{*,*}(X_k)$ is free by inductive assumption. Hence $\tilde{H}^{*,*}(X_{k+1}) \cong \tilde{H}^{*,*}(X_k) \oplus \Mtb{\nu}$ is free and there are no shifts. That is, the generators of $\tilde{H}^{*,*}(X_{k+1})$ are in the same bidegrees as the generators of $\tilde{H}^{*,*}(X_k)$, with a single new generator in bidegree $(p,q)$.
In particular, since the inductive hypothesis holds for the generators of $\tilde{H}^{*,*}(X_k)$, any free generators of $\tilde{H}^{*,*}(X_{k+1})$ with topological dimension $p$ will have weight not exceeding $q$.  Thus $\tilde{H}^{*,*}(X_{k+1})$ satisfies the inductive hypothesis.

In case $(b)$, the extension problem is again easily solved.  Recall that after the change of basis, $\tilde{H}^{*,*}(X_k) \cong \Mtb{\lambda, \chi_1, \dots, \chi_m}$ with $|\lambda| = (p-1,q)$ and $d(\chi_i) = 0$ for all $i$. As the differential is surjective, $\cok(d) = 0$ and the short exact sequence becomes
\[
0 \to \tilde{H}^{*,*}(X_{k+1}) \to \Mtb{ \chi_1, \dots, \chi_m} \to 0,
\]
so $\tilde{H}^{*,*}(X_{k+1}) \cong \Mtb{ \chi_1, \dots, \chi_m}$ is free.  Here we say $\lambda$ kills $\nu$, making $\tilde{H}^{*,*}(X_{k+1})$ the same as $\tilde{H}^{*,*}(X_{k})$, but with a single copy of $\Mt$ generated in bidegree $(p-1,q)$ removed.
Again, any free generators of $\tilde{H}^{*,*}(X_{k+1})$ with topological dimension $p$ have weight not exceeding $q$, satisfying the inductive hypothesis.

It remains to solve the extension problem arising in case $(c)$, which is significantly more labor-intensive.  Under these circumstances the usual short exact sequence
\[
0 \to \cok(d) \to \tilde{H}^{*,*}(X_{k+1}) \to \ker(d) \to 0
\]
does not split, yet we will still show that $\tilde{H}^{*,*}(X_{k+1})$ is a free $\Mt$-module.

In this case we can write $\tilde{H}^{*,*}(X_{k}) \cong \Mtb{ \omega_1, \dots, \omega_n, \chi_{n+1}, \dots, \chi_{m+1} }$ where $\omega_1, \dots, \omega_n$ satisfy the ramp condition, each $\omega_i$ supports a nonzero differential to the bottom cone $\Mt^- \langle \nu \rangle$, and $d(\chi_i) = 0$ for all $n+1 \leq i \leq m+1$.
The image under $d$ of each $\omega_i$ is
\[
d(\omega_i) = \frac{\theta}{\rho^{j_i} \tau^{k_i}} \nu
\]
for some integers $j_i, k_i \geq 0$.  It follows from the ramp condition that $j_i > j_{i+1}$ and $k_i < k_{i+1}$.  That is, each $\omega_{i+1}$ is to the right of $\omega_i$ and lies on a lower diagonal.  An example of such a differential (with the $\chi_i$ omitted) is pictured in Figure \ref{ramp diff}.

\begin{figure}[ht]
\begin{center}
\begin{tikzpicture}[scale=0.4]
\draw[gray] (-1.5,-1) -- (12.5,-1);
\draw[gray] (0,-2.5) -- (0,12.5);

\draw (11,-0.9) -- (11,-1.1) node[below] {\small $p$};
\draw (0.1,11) -- (-0.1,11) node[left] {\small $q$};
\draw (0.1,9) -- (-0.1,9) node[left] {\small $q-2$};

\draw[->,thick] (3,2) -- (3.95,2);
\draw[->,thick] (5,1) -- (5.9,1);
\draw[->,thick] (6,1) -- (6.95,1);
\draw[->,thick] (8,2) -- (8.95,2);
\draw[->,thick] (9,0) -- (9.95,0);

\draw[thick,blue] (11,11) -- (12.5,12.5);
\draw[thick,blue] (11,11) -- (11,12.5);
\draw[thick,blue] (11,9) -- (11,-2.5);
\draw[thick,blue] (11,9) -- (-0.5,-2.5);
\draw[blue] (11,11) node[below] {\small $\nu$};
\fill[blue] (11,11) circle(3.75pt);

\draw[thick,red] (3,2) -- (12.5,11.5);
\draw[thick,red] (3,2) -- (3,12.5);
\draw[thick,red] (3,0) -- (3,-2.5);
\draw[thick,red] (3,0) -- (0.5,-2.5);
\draw[red] (3,2) node[below] {\small $\omega_1$};
\fill[red] (3,2) circle(3.75pt);

\draw[thick,red] (5,1) -- (12.5,8.5);
\draw[thick,red] (5,1) -- (5,12.5);
\draw[thick,red] (5,-1) -- (5,-2.5);
\draw[thick,red] (5,-1) -- (3.5,-2.5);
\draw[red] (5,1) node[below] {\small $\omega_2$};
\fill[red] (5,1) circle(3.75pt);

\draw[thick,red] (6,1) -- (12.5,7.5);
\draw[thick,red] (6,1) -- (6,12.5);
\draw[thick,red] (6,-1) -- (6,-2.5);
\draw[thick,red] (6,-1) -- (4.5,-2.5);
\draw[red] (6,1) node[below] {\small $\omega_3$};
\fill[red] (6,1) circle(3.75pt);

\draw[thick,red] (8,2) -- (12.5,6.5);
\draw[thick,red] (8,2) -- (8,12.5);
\draw[thick,red] (8,0) -- (8,-2.5);
\draw[thick,red] (8,0) -- (5.5,-2.5);
\draw[red] (8,2) node[below] {\small $\omega_4$};
\fill[red] (8,2) circle(3.75pt);

\draw[thick,red] (9,0) -- (12.5,3.5);
\draw[thick,red] (9,0) -- (9,12.5);
\draw[thick,red] (9,-2) -- (9,-2.5);
\draw[thick,red] (9,-2) -- (8.5,-2.5);
\draw[red] (9,0) node[below] {\small $\omega_5$};
\fill[red] (9,0) circle(3.75pt);

\end{tikzpicture}
\end{center}
\caption{Ramp of differentials ($\chi_i$ omitted and $n=5$).}\label{ramp diff}
\end{figure}
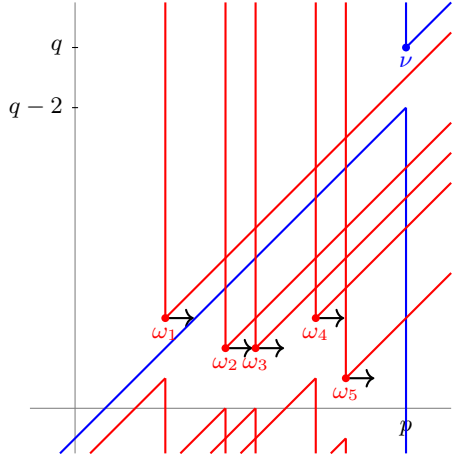

We will use $\tau$-localization to show the solution to the extension problem in this case must be free.  Since localization preserves exactness for finite sequences, we may consider the long exact sequence
\[
\cdots \rightarrow \tau^{-1}\tilde{H}^{*,*}(X_{k+1}) \rightarrow \tau^{-1}\tilde{H}^{*,*}(X_k) \xrightarrow{\tau^{-1}d} \tau^{-1}\tilde{H}^{*+1,*}(S^{p,q}) \rightarrow \cdots.
\]
Notice that $\im(d) \subseteq \Mt^-\langle \nu \rangle$ so $\tau^{-1}d \equiv 0$.  Thus this long exact sequence reduces to a short exact sequence
\[
0 \to \tau^{-1}\tilde{H}^{*,*}(S^{p,q}) \to \tau^{-1}\tilde{H}^{*,*}(X_{k+1}) \to \tau^{-1}\tilde{H}^{*,*}(X_k) \to 0,
\]
which we can rewrite as
\[
0 \to \Sigma^{p,q}\A_\infty \to \tau^{-1}\tilde{H}^{*,*}(X_{k+1}) \to \left(\bigoplus_i \Sigma^{|\omega_i|}\A_\infty\right) \oplus \left(\bigoplus_j \Sigma^{|\chi_j|}\A_\infty\right) \to 0.
\]
As modules over $\tau^{-1}\Mt \cong \A_\infty$, both the left and right terms are free.  So this short exact sequence splits and the middle term is also free.
Notice $X_{k+1}$ is also a finite $C_2$-CW complex so the structure theorem applies.  By Remark \ref{rho-torsion remark}, $\tilde{H}^{*,*}(X_{k+1})$ has no copies of $\A_n$ for any $n$, making it a free $\Mt$-module.

At this point it may seem that we are done, but in fact we still need to verify the inductive hypothesis about the topological dimensions of the generators and that the weight of any generator in dimension $p$ is no more than $q$.  The splitting of the short exact sequence above implies that all the free generators of $\tilde{H}^{*,*}(X_{k+1})$ are in the same topological dimensions as the free generators of $\tilde{H}^{*,*}(S^{p,q})$ and $\tilde{H}^{*,*}(X_{k})$.  This might seem to suggest the generators have not shifted at all, but a shift up or down would not be witnessed by $\tau$-localization.

Knowing the topological dimensions of the generators, we use $\rho$-localization to prove the constraint on the weight of any generators in dimension $p$.  First, we split off the less interesting free generators corresponding to the $\chi_i \in \tilde{H}^{*,*}(X_k)$.  The $\chi_i$ are elements of $\ker(d)$ so we can lift them back to $\tilde{H}^{*,*}(X_{k+1})$.  Since $d(\chi_i)=0$, by exactness we can choose $b_i$ in the preimage $(\iota^*)^{-1}(\chi_i)$, though this choice is not necessarily unique.  It is fairly easy to see that $\theta b_i \neq 0$ for all $i$.
This follows because $\Mtb{\chi_{n+1}, \dots, \chi_{m+1}}$ is a free submodule of $\tilde{H}^{*,*}(X_k)$ so $\theta \chi_i \neq 0$ for each $i$.  Since $\iota^*$ is an $\Mt$-module map, $\iota^*(\theta b_i) = \theta \iota^*(b_i) = \theta \chi_i \neq 0$.  So it must be that $\theta b_i \neq 0$.
Then by Lemma \ref{theta lemma} each $\Mtb{b_i}$ includes as a free submodule of $\tilde{H}^{*,*}(X_{k+1})$. Using the independence of the $\chi_i$, we can take all the  $b_i$ together and include $\Mt\langle b_{n+1}, \dots, b_{m+1} \rangle$ as a free submodule.

We now have a short exact sequence
\[
0 \to \Mtb{b_{n+1}, \dots, b_{m+1}} \to \tilde{H}^{*,*}(X_{k+1}) \to N \to 0,
\]
where $N$ is the quotient.  By Lemma \ref{injectivity lemma}, we know that $\Mt$ is self-injective.  So this short exact sequence splits, making $N$ a direct summand of $\tilde{H}^{*,*}(X_{k+1})$.  Moreover, since we already know $\tilde{H}^{*,*}(X_{k+1})$ is free, $N$ is also free.\footnote{Note that $\Mt$ is a graded local ring with unique maximal ideal $(\rho,\tau)$.  Thus every finitely generated projective $\Mt$-module is free.
See \cite{Huishi}.}

Now we will $\rho$-localize the original long exact sequence.  As with $\tau$-localization, $\rho$-localization is exact here and since $\im(d) \subseteq \Mt^-\langle \nu \rangle$ we have that $\rho^{-1}d \equiv 0$.  So we have the short exact sequence
\[
0 \to \rho^{-1}\tilde{H}^{*,*}(S^{p,q}) \to \rho^{-1}\tilde{H}^{*,*}(X_{k+1}) \to \rho^{-1}\tilde{H}^{*,*}(X_k) \to 0.
\]
Into this short exact sequence we can include the $\rho$-localization of the isomorphism $\bigoplus\Mt\langle b_i\rangle\iso\bigoplus\Mt\langle \chi_i\rangle$, and quotient to obtain a short exact sequence of short exact sequences as in the diagram below.  Everything in the following diagram is a free $\rho^{-1}\Mt$-module so all the short exact sequences split.
\begin{center}
\begin{tikzcd}
 & 0 \arrow[d] & 0 \arrow[d] & 0 \arrow[d] & \\
0 \arrow[r] & 0 \arrow[r] \arrow[d] & \bigoplus_i (\rho^{-1}\Mtb{b_i}) \arrow[r] \arrow[d] & \bigoplus_i (\rho^{-1}\Mtb{\chi_i}) \arrow[r] \arrow[d] & 0\\
0 \arrow[r] & \rho^{-1}\tilde{H}^{*,*}(S^{p,q}) \arrow[r] \arrow[d] & \rho^{-1}\tilde{H}^{*,*}(X_{k+1}) \arrow[r] \arrow[d] & \rho^{-1}\tilde{H}^{*,*}(X_k) \arrow[r] \arrow[d] & 0\\
0 \arrow[r] & \rho^{-1}\tilde{H}^{*,*}(S^{p,q}) \arrow[r] \arrow[d] & \rho^{-1}N \arrow[r] \arrow[d] & \bigoplus_i (\rho^{-1}\Mtb{\omega_i}) \arrow[r] \arrow[d] & 0\\
 & 0 & 0 & 0 & \\
\end{tikzcd}
\end{center}
In particular, from the last row we have that
\[
\rho^{-1}N \cong \left(\rho^{-1}\tilde{H}^{*,*}(S^{p,q})\right) \oplus \left(\bigoplus_{i=1}^{n} (\rho^{-1}\Mtb{\omega_i})\right).
\]
Thus, making use of Remark \ref{rho-torsion remark}, the fixed-set dimensions of the free generators of $N$ are the same as the fixed-set dimensions of the free generators of $\Mtb{\omega_1, \dots, \omega_n}$ together with that of $\tilde{H}^{*,*}(S^{p,q})$.

Again, it seems plausible at this point that we simply have a new generator in the bidegree of $\nu$ with no shifts.
Regardless, we claim any free generator of $N$ in topological dimension $p$ must have fixed-set dimension at least $p-q$ and so have weight at most $q$ (but possibly lower).  This is because each of the $\omega_i$ maps to the lower cone of $\nu$, so $\fix(\nu) = p-q$ is the lowest fixed-set dimension of the generators of $N$.  The generators $b_i$ are in precisely the same bidegrees as the original generators $\chi_i$.  Thus all free generators of $\tilde{H}^{*,*}(X_{k+1})$ having topological dimension $p$ have weight at most $q$.  This completes the inductive step and hence the proof of the freeness theorem for finite $\Rep(C_2)$-complexes.
\end{proof}

\subsection{Finite type freeness theorem}

As a consequence of the finite case, we can extend Kronholm's freeness theorem to infinite complexes of finite type. Recall, we say that a $\text{Rep}(C_2)$-complex is finite type if it is built with finitely many cells of each fixed-set dimension.

\begin{thm}
If $X$ is a finite type $\Rep(C_2)$-complex then $\tilde{H}^{*,*}(X;\underline{\F_2})$ is free as a graded $\Mt$-module.
\end{thm}

\begin{proof}
As in the proof of Theorem \ref{main thm}, we begin by filtering $X$ as a sequence of subcomplexes $X_k$ where each $X_{k+1}$ is formed from $X_k$ by attaching a single $\Rep(C_2)$-cell $e_{k+1}$.  We assume the cells are attached by increasing topological dimension and, within each topological dimension, by increasing weight.  We will show the $\varprojlim^1$ term vanishes and find a free basis for the inverse limit.

Because $X$ is finite type, for each $i \geq 0$ we can choose a subcomplex $X_{k(i)}$ that contains every cell of fixed-set dimension less than or equal to $i$ by defining
\[
k(i)=\max\{k: e_k\cong D(\R^{p_k,q_k})\text{ and }p_k-q_k\le i\}.
\]
Note that $\tilde{H}^{*,*}(X_{k(i)})$ is free by Theorem \ref{main thm} because $X_{k(i)}$ is a finite $\Rep(C_2)$-complex.  Recall from the proof of Theorem \ref{main thm} that the $\Mt$ generators of the cohomology may lie in different bidegrees than the cells forming $X_{k(i)}$ and there may be fewer $\Mt$ generators than cells.  However, each $\Mt$ generator of the cohomology will have the same fixed-set dimension as one of the cells used in constructing the space.

Now we will consider the cofiber sequence $X_{k(i)} \to X \to X/X_{k(i)}$.  As usual, this induces a long exact sequence in cohomology
\[
\cdots \to \tilde{H}^{*,*}\left(X/X_{k(i)}\right) \to \tilde{H}^{*,*}(X) \to \tilde{H}^{*,*}\left(X_{k(i)}\right) \to \tilde{H}^{*+1,*}\left(X/X_{k(i)}\right) \to \cdots.
\]
Observe that, aside from the basepoint, all the cells of the quotient $X/X_{k(i)}$ have fixed-set dimension greater than $i$.  Thus the reduced cohomology of the quotient vanishes for all bidegrees $(p,q)$ with $p \leq i$ and $q \geq p-i-2$ as in Figure \ref{quotient vanishing fig}.

\begin{figure}[ht]
\begin{tikzpicture}[scale=0.375]

\draw[gray] (-1.5,-1) -- (9.5,-1);
\draw[gray] (0,-7.5) -- (0,5.5);

\fill[lightgray, opacity=0.25] (4,5.5) -- (4,-2) -- (-1.5,-7.5) -- (-1.5,5.5) -- cycle;
\draw[thick, gray] (4,5.5) -- (4,-2) -- (-1.5,-7.5);

\draw[thick, gray, dashed] (4,-1) -- (9.5,4.5);

\draw (0.1,-6) -- (-0.1,-6) node[right] {\small $\ i-2$};
\draw (9.5,-1) node[below, black] {\small $p$};
\draw (0,5.5) node[left, black] {\small $q$};
\draw (4,-1) node[above left] {\small $i$};

\end{tikzpicture}
\caption{Vanishing region for $\tilde{H}^{*,*}\left(X/X_{k(i)}\right)$.}\label{quotient vanishing fig}
\end{figure}
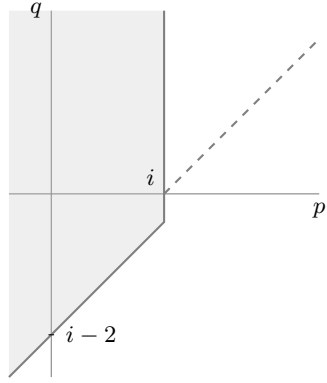

Fix any bidegree $(p,q)$.  There exists an $i$ such that $i > p$ and $i > p-q-2$.  This implies both $(p,q)$ and $(p+1,q)$ lie in the vanishing region for $X/X_{k(i)}$.  So in the long exact sequence for the cofiber sequence $X_{k(i)} \to X \to X/X_{k(i)}$ we have
\[
\cdots \to 0 \to \tilde{H}^{p,q}(X) \to \tilde{H}^{p,q}\left(X_{k(i)}\right) \to 0 \to \cdots
\]
and we see that $\tilde{H}^{p,q}(X) \cong \tilde{H}^{p,q}\left(X_{k(i)}\right)$.  Similarly, if $n>k(i)$ and we consider the cofiber sequence $X_{k(i)} \to X_{n} \to X_{n}/X_{k(i)}$, we have that $\tilde{H}^{p,q}(X_n) \cong \tilde{H}^{p,q}\left(X_{k(i)}\right)$.
This tells us that, as a vector space, every bidegree of the cohomology of $X$ stabilizes by some finite stage $k(i)$.  Of course, this value depends on both $p$ and $q$.

Recall, as $X = \colim X_k$, we have the Milnor exact sequence
\[
0\to {\varprojlim}^1 \tilde{H}^{*-1,*} (X_k)\to \tilde{H}^{*,*} (X)\to \varprojlim \tilde{H}^{*,*}(X_k)\to 0,
\]
where the inverse limit is taken for the tower
\[
\tilde{H}^{*,*}(X_{0}) \xleftarrow{\iota^*}
\tilde{H}^{*,*}(X_{1}) \xleftarrow{\iota^*} \cdots
\xleftarrow{\iota^*}
\tilde{H}^{*,*}(X_{k})
 \xleftarrow{\iota^*}
 \tilde{H}^{*,*}(X_{k+1})
  \xleftarrow{\iota^*} \cdots
\]
(see for example Section 3F of \cite{Hatcher}).
In particular, for each $p$ and $q$ we have the short exact sequence
\[
0\to {\varprojlim}^1 \tilde{H}^{p-1,q} (X_k)\to \tilde{H}^{p,q} (X)\to \varprojlim \tilde{H}^{p,q}(X_k)\to 0.
\]
Fix $(p,q)$ and choose $i$ such that $i > p$ and $i>p-q-2$ as above.  Truncating the sequence does not change the value of $\varprojlim^1$ or the inverse limit, so we have
\[
\underset{\scriptscriptstyle{k}}{\varprojlim}^1 \tilde{H}^{p-1,q} (X_k) \iso \underset{\scriptscriptstyle{n \ge k(i)}}{\varprojlim}\!^{1} \tilde{H}^{p-1,q} \left(X_{n}\right)
\]
and
\[
\underset{\scriptscriptstyle{k}}\varprojlim\, \tilde{H}^{p,q} (X_k) \cong \underset{\scriptscriptstyle{n\ge k(i)}}\varprojlim \tilde{H}^{p,q} \left(X_n\right).
\]
As shown above, the maps in the truncated tower
\[\H pq \left(X_{k(i)}\right) \xleftarrow{\iota^*} \H pq \left(X_{k(i)+1}\right) \xleftarrow{\iota^*} \H pq \left(X_{k(i)+2}\right) \xleftarrow{\iota^*} \dots\]
are all isomorphisms.

Recall that $\varprojlim^1$ vanishes whenever all the maps in the tower are surjective and thus  ${\varprojlim}^1 \tilde{H}^{p-1,q} \left(X_k\right) = 0$.  The same argument can be made for any bidegree $(p,q)$, hence ${\varprojlim}^1 \tilde{H}^{*-1,*} (X_{k}) \equiv 0$, making $\tilde{H}^{*,*}(X) \to \varprojlim \tilde{H}^{*,*}(X_k)$ an isomorphism of bigraded vectors spaces.  Since this is an $\Mt$-module map,  $\tilde{H}^{*,*}(X) \cong \varprojlim \tilde{H}^{*,*}(X_k)$ as $\Mt$-modules.

It remains to show that $\varprojlim \tilde{H}^{*,*}(X_k)$ is free.  For that we use a stabilization phenomenon of the cohomologies as $\Mt$-modules to find a basis for the inverse limit.  Fix an $i \geq 0$ and consider $\tilde{H}^{*,*}\left(X_{k(i+1)}\right)$ as an $\Mt$-module.  If we attach the next cell to form $X_{k(i+1)+1}$ and compute its cohomology via the usual cofiber sequence, any $\Mt$ generators with fixed-set dimension $i$ will be sent to zero by the differential in the long exact sequence. This is simply for degree reasons, as an $\Mt$ generator with fixed-set dimension $i$ cannot support a nonzero differential to an $\Mt$ generated by an element with fixed-set dimension $i+2$ or greater.  Furthermore, the generators of fixed-set dimension $i$ will continue to support zero differentials when computing the cohomology of $X_n$ for any $n > k(i+1)$.

Now for each $i\ge 0$, choose a graded free basis for $\tilde{H}^{*,*}\left(X_{k(i+1)}\right)$ and let $\beta_i$ be a subset of that basis consisting only of generators of fixed-set dimension $i$.
The union of all such $\beta_i$ will form our graded free basis for $\tilde{H}^{*,*}(X)$.  For any index $i$, we can define a map $\Mtb{\beta_i} \to \varprojlim\tilde{H}^{*,*}(X_k)$ by including $\Mtb{\beta_i} \to \tilde{H}^{*,*}\left(X_{k(i+1)}\right)$.
In fact, we have an inclusion $\Mtb{\beta_i} \to \tilde{H}^{*,*}(X_{n})$ for any $n > k(i+1)$ since if a generator in $\beta_i$ has bidegree $(a,b)$, then $\tilde{H}^{a,b}\left(X_{k(i+1)}\right) \cong \tilde{H}^{a,b}(X_{n})$. 
Thus we can define a cone over the inverse system
\begin{center}
\begin{tikzcd}
\tilde{H}^{*,*}\left(X_{k(i+1)}\right)
& \arrow[swap]{l}{\iota^*}\tilde{H}^{*,*}\left(X_{k(i+1)+1}\right)
& \arrow[swap]{l}{\iota^*}\tilde{H}^{*,*}\left(X_{k(i+1)+2}\right)
& \arrow[swap]{l}{\iota^*}\cdots\\
 & \Mtb{\beta_i} \arrow{ul} \arrow{u} \arrow{ur}
\end{tikzcd}
\end{center}
and so we get a map to the inverse limit
$$ f_i : \Mtb{\beta_i} \to \varprojlim_j\tilde{H}^{*,*}\left(X_{k(i+1)+j}\right) \iso \varprojlim\tilde{H}^{*,*}(X_k).$$

Finally taking the direct sum over all $f_i$ we have a map
\[
f: \bigoplus_{i=0 }^{\infty}\Mtb{\beta_i} \longrightarrow \varprojlim\tilde{H}^{*,*}(X_k) \cong \tilde{H}^{*,*}(X).
\]
As above, for any bidegree $(p,q)$, the map $f$ factors through a finite stage where $\tilde{H}^{p,q} \left(X_{k(i)}\right)\iso \tilde{H}^{p,q}(X)$ for some appropriate choice of $i$. Thus the $\Mt$-module map $f$ is an isomorphism because it is an isomorphism in every bidegree, and hence $H^{*,*}(X)$ is free.
\end{proof}

\begin{counterexample}\label{counterex}
The freeness theorem does not hold for all locally finite $\Rep(C_2)$-complexes.  Consider the space $Y = \bigvee_{n=0}^{\infty} S^{n,n}$.  By the wedge axiom
\[
\tilde{H}^{*,*}(Y) \cong \prod_{n=0}^{\infty} \tilde{H}^{*,*}(S^{n,n}) \cong \prod_{n=0}^\infty \Sigma^{n,n} \Mt.
\]
Nonequivariantly, the singular cohomology of the underlying space is a free $\F_2$-module because $\F_2$ is a field.  However, $\Mt$ is not a field and $\tilde{H}^{*,*}(Y)$ is not a free $\Mt$-module.  Suppose $\gamma_n$ is the generator of the $\Mt$ in bidegree $(n,n)$ of the infinite product.  Then there is an element in bidegree $(0,-2)$ of $\tilde{H}^{*,*}(Y)$ of the form
\[
x = \left( \theta \gamma_0, \frac{\theta}{\rho}\gamma_1, \frac{\theta}{\rho^2}\gamma_2, \frac{\theta}{\rho^3}\gamma_3, \dots \right).
\]
Notice that $x$ is not $\rho$-torsion since for any $k$ we have $\rho^k x \neq 0$. However $\tau x = 0$.  In $\Mt$, and indeed in any free module, an element that is not $\rho$-torsion is also not $\tau$-torsion.  Hence $\tilde{H}^{*,*}(Y)$ cannot be free.
\end{counterexample}

As a corollary to the freeness theorem, we have a splitting at the spectrum level.

\begin{cor}
Suppose $X$ is a finite type $\Rep(C_2)$-complex.  Then
\[
F(\Sigma^{\infty}X_+, H\underline{\F_2}) \simeq \bigvee_{i \in I} \left(S^{-p_i,-q_i} \Smash H\underline{\F_2}\right),
\]
where $I$ is a countable set indexing the elements in any free basis for $H^{*,*}(X)$.
\end{cor}

\begin{proof}
Let $\beta$ be a free basis for $H^{*,*}(X)$ with each generator $\gamma_i$ having bidegree $(p_i,q_i)$ for $i \in I$.  By adjunction $H^{*,*}(X) \cong \pi_{-*,-*}F(\Sigma^{\infty} X_+, H \underline{\F_2})$ via
\begin{align*}
H^{p,q}(X) &= [\Sigma^{\infty}X_+, S^{p,q}\Smash H\underline{\F_2}]\\
&\cong [\Sigma^{\infty}X_+ \Smash S^{-p,-q}, H\underline{\F_2}]\\
&\cong [S^{-p,-q}, F(\Sigma^{\infty}X_+, H\underline{\F_2})]\\
&= \pi_{-p,-q}F(\Sigma^{\infty}X_+, H\underline{\F_2}).
\end{align*}
Each $\gamma_i \in \pi_{-p_i,-q_i}F(\Sigma^{\infty} X_+, H \underline{\F_2})$ gives rise to a map of spectra
\[
S^{-p_i, -q_i} \Smash H\underline{\F_2} \xrightarrow{\gamma_i \Smash id} F(\Sigma^{\infty} X_+, H \underline{\F_2}) \Smash H\underline{\F_2} \xrightarrow{\mu} H\underline{\F_2},
\]
where $F(\Sigma^{\infty} X_+, H \underline{\F_2})$ is an $H\underline{\F_2}$-module via the structure map $\mu$.
Taking all of these maps together, we get a map
\[
\bigvee_{i\in I} \left(S^{-p_i,-q_i} \Smash H\underline{\F_2} \right) \longrightarrow F(\Sigma^{\infty}X_+, H\underline{\F_2}).
\]
By construction this map induces an isomorphism on bigraded homotopy.  Since $RO(C_2)$-graded homotopy groups detect weak equivalences (see for example Section 6.10 of \cite{CM}), the map is a weak equivalence.
\end{proof}

\section{Kronholm shifts of free generators}\label{shifts}

In this section we specify the shifting of free generators that occurs when attaching a single cell to a $\Rep(C_2)$-complex with an interesting differential.  By interesting differential, we mean a differential of the last type that shows up in the proof of the freeness theorem in Section \ref{main thm}, where a ramp (Definition \ref{rampdef}) of generators map to the bottom cone.  The setup of the following theorem is the same as case (c) in the proof, with $X$ playing the role of $X_k$ and $Y$ playing the role of $X_{k+1}$.  The result is that we can precisely specify the bidegrees of free generators in cohomology after such a differential.

\begin{thm}\label{formula}
Let $X$ be a finite $\Rep(C_2)$-complex of dimension less than or equal to $p$.  Let $Y$ be the space formed by attaching a single $\Rep(C_2)$-cell $D(\R^{p,q})$ to $X$ where any cells of $X$ with dimension $p$ have weight less than or equal to $q$.  The cofiber sequence $X \to Y \to Y/X \cong S^{p,q}$ induces a long exact sequence in reduced cohomology with differential $d: \tilde{H}^{*,*}(X) \to \tilde{H}^{*+1,*}(S^{p,q})$.

Assume there exists a basis for $\tilde{H}^{*,*}(X) \cong \Mtb{\omega_1, \dots, \omega_n, \chi_{n+1}, \dots, \chi_{m+1}}$, where $\omega_1, \dots, \omega_n$ satisfy the ramp condition and let $\nu$ denote the free generator of $\tilde{H}^{*,*}(S^{p,q})$.  Suppose each $\omega_i$ supports a nonzero differential to $\Mt^{-}\langle \nu \rangle$ with
\[
d(\omega_i) = \frac{\theta}{\rho^{j_i}\tau^{k_i}} \nu
\]
for some integers $j_i, k_i \geq 0$ and $d(\chi_i) = 0$ for $n+1 \leq i \leq m+1.$  Then
\[
\tilde{H}^{*,*}(Y) \cong
\left(\bigoplus_{i=1}^n \Sigma^{|\omega_i| + (0,s_i)} \Mt\right)
\oplus \left(\Sigma^{|\nu| - (0,s_1+s_2+\dots+s_n)} \Mt\right)
\oplus \left(\bigoplus_{j=n+1}^{m+1} \Sigma^{|\chi_j|} \Mt\right)
\]
where the shifts in weight are given by
\begin{align*}
	s_1& = k_1 + 1\\
	s_2&=k_2-k_1\\
	s_3&=k_3-k_2\\
	&\ \, \vdots\\
	s_n&=k_n-k_{n-1}.
\end{align*}
\end{thm}

\begin{figure}[ht]
\begin{center}
\begin{tikzpicture}[scale=0.4]
\draw[gray] (-1.5,-1) -- (12.5,-1);
\draw[gray] (0,-2.5) -- (0,12.5);

\draw (11,-0.9) -- (11,-1.1) node[below] {\small $p$};
\draw (0.1,11) -- (-0.1,11) node[left] {\small $q$};

\draw[->,gray,dashed] (3,2) -- node[black,left]{\small $s_1$} (3,3);
\draw[->,gray,dashed] (5,1) -- node[black,above left]{\small $s_2$} (5,4);
\draw[->,gray,dashed] (6,1) -- node[black,right]{\small $s_3$} (6,2);
\draw[->,gray,dashed] (8,2) -- node[black,left]{\small $s_4$} (8,3);
\draw[->,gray,dashed] (9,0) -- node[black,right]{\small $s_5$} (9,3);
\draw[->,gray,dashed] (11,11) -- node[black,right]{\small $\Sigma s_i$} (11,2);

\fill (3,2) circle(3.75pt) node[below] {\small $\omega_1$};
\fill (5,1) circle(3.75pt) node[below] {\small $\omega_2$};
\fill (6,1) circle(3.75pt) node[below] {\small $\omega_3$};
\fill (8,2) circle(3.75pt) node[below] {\small $\omega_4$};
\fill (9,0) circle(3.75pt) node[below] {\small $\omega_5$};
\fill (11,11) circle(3.75pt) node[above] {\small $\nu$};

\end{tikzpicture}\hfill
\begin{tikzpicture}[scale=0.4]
\draw[gray] (-1.5,-1) -- (12.5,-1);
\draw[gray] (0,-2.5) -- (0,12.5);

\draw (11,-0.9) -- (11,-1.1) node[below] {\small $p$};
\draw (0.1,11) -- (-0.1,11) node[left] {\small $q$};

\fill (3,3) circle(3.75pt);
\fill (5,4) circle(3.75pt);
\fill (6,2) circle(3.75pt);
\fill (8,3) circle(3.75pt);
\fill (9,3) circle(3.75pt);
\fill (11,2) circle(3.75pt);

\draw (3,3) node[below] {\small $a_0$};
\draw (5,4) node[below] {\small $a_1$};
\draw (6,2) node[below] {\small $a_2$};
\draw (8,3) node[below] {\small $a_3$};
\draw (9,3) node[below] {\small $a_4$};
\draw (11,2) node[below] {\small $a_5$};

\end{tikzpicture}
\end{center}
\caption{Example of Kronholm shifts predicted by Theorem \ref{formula}, corresponding to Figures \ref{ramp diff}, \ref{lattice}, and \ref{ramp diff three}.}
\end{figure}
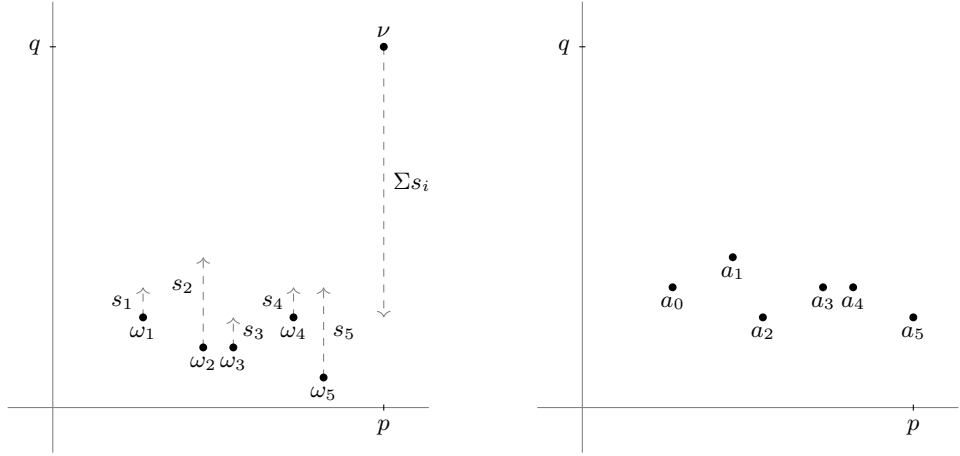

\begin{proof}
The cofiber sequence $X \to Y \to S^{p,q}$ induces a long exact sequence in cohomology
\[
\cdots \xrightarrow{d} \tilde{H}^{*,*}(S^{p,q}) \xrightarrow{\pi^*} \tilde{H}^{*,*}(Y) \xrightarrow{\iota^*} \tilde{H}^{*,*}(X) \xrightarrow{d} \tilde{H}^{*+1,*}(S^{p,q}) \xrightarrow{\pi^*} \cdots
\]
as usual.

In the proof of Theorem \ref{main thm}, we investigated the short exact sequence
\[
0 \to \cok(d) \to \tilde{H}^{*,*}(Y) \to \ker(d) \to 0.
\]
Recall from part (c) of the proof, using $\tau$-localization we know the topological dimensions of the free generators of $\tilde{H}^{*,*}(Y)$ are the same as those of $\tilde{H}^{*,*}(X)$ together with $\tilde{H}^{*,*}(S^{p,q})$.  Similarly, we know from $\rho$-localization that the fixed-set dimensions of the free generators are preserved.  Furthermore, we know that the generator in topological dimension $p$ of $\tilde{H}^{*,*}(Y)$ that ``corresponds'' to $\nu$ has fixed-set dimension at least $p-q$.  We also had the following split short exact sequence
\[
0 \to \Mtb{b_{n+1}, \dots, b_{m+1}} \to \tilde{H}^{*,*}(Y) \to N \to 0,
\]
where $\iota^*(b_i) = \chi_i$, making the quotient $N$ a free direct summand.

Including the isomorphism $\bigoplus\Mt\langle b_i\rangle\iso\bigoplus\Mt\langle \chi_i\rangle$ and quotienting, we have a short exact sequence of short exact sequences as below.
\begin{center}
\begin{tikzcd}
 & 0 \arrow[d] & 0 \arrow[d] & 0 \arrow[d] & \\
0 \arrow[r] & 0 \arrow[r] \arrow[d] & \bigoplus_i \Mtb{b_i} \arrow[r] \arrow[d] & \bigoplus_i \Mtb{\chi_i} \arrow[r] \arrow[d] & 0\\
0 \arrow[r] & \cok(d) \arrow[r] \arrow[d] & \tilde{H}^{*,*}(Y) \arrow[r] \arrow[d] & \ker(d) \arrow[r] \arrow[d] & 0\\
0 \arrow[r] & \cok(\tilde{d}) \arrow[r] \arrow[d] & N \arrow[r] \arrow[d] & \ker(\tilde{d}) \arrow[r] \arrow[d] & 0\\
 & 0 & 0 & 0 &
\end{tikzcd}
\end{center}
Here $\tilde{d}$ is the restriction of the differential to the direct summand $\Mtb{\omega_1, \dots, \omega_n}$.  An example of the restricted differential $\tilde{d}$ was shown in Figure \ref{ramp diff}.  For ease of reference, the same example appears again on the left in Figure \ref{ramp diff three}.

Now fix $i$ and consider an element $x \in N$ with bidegree $|x| = |\omega_i|$.  As a consequence of the ramp condition and the fact that each $\omega_i$ supports a nonzero differential, the kernel of $\tilde{d}$ in bidegree $|\omega_i|$ is always zero. So $x$ must be sent to zero in $\ker(\tilde{d})$ and thus lift to an element of $\cok(\tilde{d})$.  Because $\im(d) \subseteq \Mt^{-}\langle\nu\rangle$, any element of $\cok(d)$ in this bidegree has $\tau$-torsion.

We know there is a single free generator of $N$ with topological dimension $\Top(\omega_i)$, but no element of $\Mt^{+}$ has $\tau$-torsion.  So this free generator of $N$ must have higher weight (and so lower fixed-set dimension).  This same argument holds for each $\omega_i$.  So every generator in these topological dimensions must have ``shifted up,'' that is, it lies at a higher weight.  In order to preserve fixed-set dimensions, the generator in topological degree $p$ that corresponds to $\nu$ must have ``shifted down.''
\end{proof}

To clarify with an example, Figure \ref{lattice} shows a slanted grid corresponding to the topological and fixed-set dimensions of the generators for the differential pictured in Figure \ref{ramp diff} (and again on the left in Figure \ref{ramp diff three}).  On the left of Figure \ref{lattice} we see the original positions of the generators and on the right the positions of the shifted generators.  The red generators correspond to $\omega_1, \dots, \omega_n$ and the blue generator corresponds to $\nu$.  A combinatorial argument allows us to identify the positions of the shifted generators shown on the right.

As there is at most one generator in each topological dimension and each fixed-set dimension, we must have a single generator on each vertical line and a single generator on each diagonal line, both before and after the differential.  We have just argued above that all the red generators, those corresponding to $\omega_1, \dots, \omega_n$, must shift up. If we consider the generators from left to right, each will have a unique unoccupied diagonal above it on the slanted grid.  So finally the blue generator corresponding to $\nu$ must shift down as pictured in Figure \ref{lattice}.

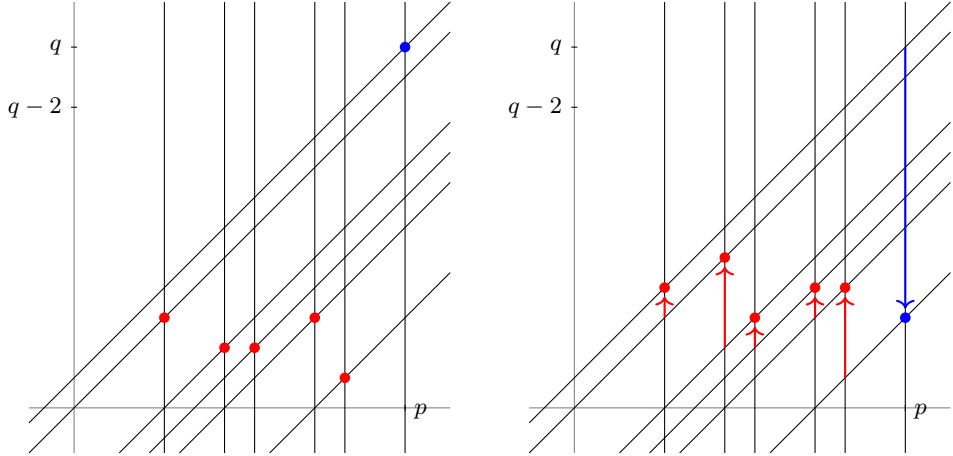
\begin{figure}[ht]
\begin{center}
\begin{tikzpicture}[scale=0.4]
\draw[gray] (-1.5,-1) -- (12.5,-1);
\draw[gray] (0,-2.5) -- (0,12.5);

\draw (11,-0.9) -- (11,-1.1) node[right] {\small $p$};
\draw (0.1,11) -- (-0.1,11) node[left] {\small $q$};
\draw (0.1,9) -- (-0.1,9) node[left] {\small $q-2$};

\draw[black] (-1.5,-1.5) -- (12.5,12.5);
\draw[black] (11,-2.5) -- (11,12.5);
\fill[blue] (11,11) circle(5pt);

\draw[black] (-1.5,-2.5) -- (12.5,11.5);
\draw[black] (3,-2.5) -- (3,12.5);
\fill[red] (3,2) circle(5pt);

\draw[black] (1.5,-2.5) -- (12.5,8.5);
\draw[black] (5,-2.5) -- (5,12.5);
\fill[red] (5,1) circle(5pt);

\draw[black] (2.5,-2.5) -- (12.5,7.5);
\draw[black] (6,-2.5) -- (6,12.5);
\fill[red] (6,1) circle(5pt);

\draw[black] (3.5,-2.5) -- (12.5,6.5);
\draw[black] (8,-2.5) -- (8,12.5);
\fill[red] (8,2) circle(5pt);

\draw[black] (6.5,-2.5) -- (12.5,3.5);
\draw[black] (9,-2.5) -- (9,12.5);
\fill[red] (9,0) circle(5pt);

\end{tikzpicture}\hfill
\begin{tikzpicture}[scale=0.4]
\draw[gray] (-1.5,-1) -- (12.5,-1);
\draw[gray] (0,-2.5) -- (0,12.5);

\draw (11,-0.9) -- (11,-1.1) node[right] {\small $p$};
\draw (0.1,11) -- (-0.1,11) node[left] {\small $q$};
\draw (0.1,9) -- (-0.1,9) node[left] {\small $q-2$};

\draw[black] (-1.5,-1.5) -- (12.5,12.5);
\draw[black] (11,-2.5) -- (11,12.5);

\draw[black] (-1.5,-2.5) -- (12.5,11.5);
\draw[black] (3,-2.5) -- (3,12.5);
\fill[red] (3,3) circle(5pt);

\draw[black] (1.5,-2.5) -- (12.5,8.5);
\draw[black] (5,-2.5) -- (5,12.5);
\fill[red] (5,4) circle(5pt);

\draw[black] (2.5,-2.5) -- (12.5,7.5);
\draw[black] (6,-2.5) -- (6,12.5);
\fill[red] (6,2) circle(5pt);

\draw[black] (3.5,-2.5) -- (12.5,6.5);
\draw[black] (8,-2.5) -- (8,12.5);
\fill[red] (8,3) circle(5pt);

\draw[black] (6.5,-2.5) -- (12.5,3.5);
\draw[black] (9,-2.5) -- (9,12.5);
\fill[red] (9,3) circle(5pt);

\fill[blue] (11,2) circle(5pt);
\color{blue}
\draw[->,thick] (11,11)--(11,2.3);
\color{red}
\draw[->,thick] (3,2)--(3,2.7);
\draw[->,thick] (5,1)--(5,3.7);
\draw[->,thick] (6,1)--(6,1.7);
\draw[->,thick] (8,2)--(8,2.7);
\draw[->,thick] (9,0)--(9,2.7);

\end{tikzpicture}
\end{center}
\caption{Shifts along the slanted grid.}\label{lattice}
\end{figure}

In general, since we have shown each $\omega_i$ shifts up, again if we consider each topological dimension from left to right, each generator has a unique position available to it on the slanted grid. These shifts of the $\omega_i$ and $\nu$ are precisely as calculated in the theorem statement.

\textbf{Identifying shifted generators.} So far we have shown the free generators ``shift'' in the case of a ramp differential, but the new generators remain a bit mysterious.  For the interested reader, we will precisely identify a choice of basis.  To do so we pick out a collection of distinguished elements in the kernel of the differential.  These are the lowest-weight elements in $\Mt^+\langle \omega_1, \dots, \omega_n \rangle$ in each of the topological degrees $\Top(\omega_1), \dots, \Top(\omega_n), \Top(\nu)$ supporting zero differentials.  We will use these to define elements in $\tilde{H}^{*,*}(Y)$.

The following chart is useful for reference as we define these elements.
\begin{center}\label{first def of ais}
\begin{tabular}{cccccc}
$\tilde{H}^{*,*}(Y)$
& $\xrightarrow{\iota^*}$
& $\tilde{H}^{*,*}(X)$
& $\xrightarrow{d}$
& $\tilde{H}^{*+1,*}(S^{p,q})$ \\
\hline \\
	$a_0$
	& $\mapsto$
	& $\tau^{k_1 +1}\omega_1$
	& $\mapsto$
	& 0 \\
		$a_i$ & $\mapsto$
		& $\rho^{j_i-j_{i+1}}\omega_i+\tau^{k_{i+1}-k_i}\omega_{i+1}$
		& $\mapsto$
		& 0 \\
			$a_n$
			& $\mapsto$
			& $\rho^{j_n +1}\omega_n$
			& $\mapsto$
			& 0 \\ \\
				$b_i$
				& $\mapsto$
				& $\chi_i$
				& $\mapsto$
				& 0 \\
\end{tabular}
\end{center}
We revisit our sample (restricted) differential in Figure \ref{ramp diff three}, now indicating the $a_i$ on the right.
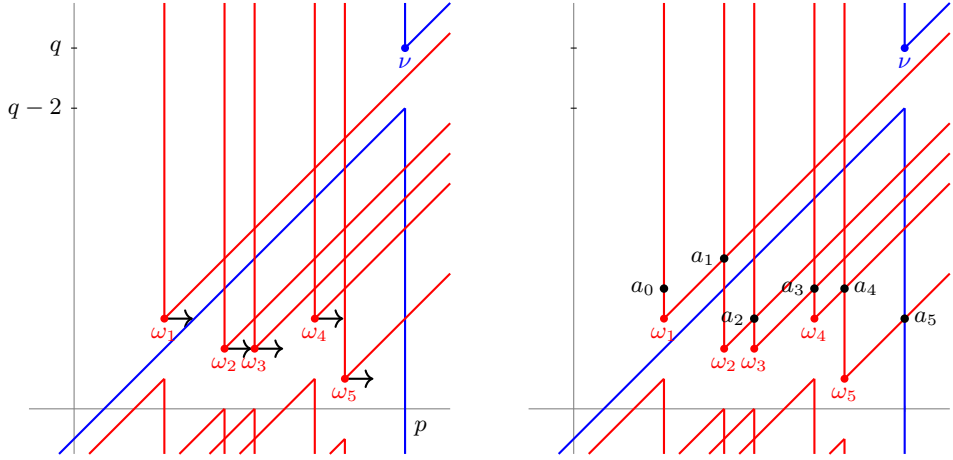
\begin{figure}[ht]
\begin{center}
\begin{tikzpicture}[scale=0.4]
\draw[gray] (-1.5,-1) -- (12.5,-1);
\draw[gray] (0,-2.5) -- (0,12.5);

\draw (11,-0.9) -- (11,-1.1) node[below right] {\small $p$};
\draw (0.1,11) -- (-0.1,11) node[left] {\small $q$};
\draw (0.1,9) -- (-0.1,9) node[left] {\small $q-2$};

\draw[->,thick] (3,2) -- (3.95,2);
\draw[->,thick] (5,1) -- (5.9,1);
\draw[->,thick] (6,1) -- (6.95,1);
\draw[->,thick] (8,2) -- (8.95,2);
\draw[->,thick] (9,0) -- (9.95,0);

\draw[thick,blue] (11,11) -- (12.5,12.5);
\draw[thick,blue] (11,11) -- (11,12.5);
\draw[thick,blue] (11,9) -- (11,-2.5);
\draw[thick,blue] (11,9) -- (-0.5,-2.5);
\draw[blue] (11,11) node[below] {\small $\nu$};
\fill[blue] (11,11) circle(3.75pt);

\draw[thick,red] (3,2) -- (12.5,11.5);
\draw[thick,red] (3,2) -- (3,12.5);
\draw[thick,red] (3,0) -- (3,-2.5);
\draw[thick,red] (3,0) -- (0.5,-2.5);
\draw[red] (3,2) node[below] {\small $\omega_1$};
\fill[red] (3,2) circle(3.75pt);

\draw[thick,red] (5,1) -- (12.5,8.5);
\draw[thick,red] (5,1) -- (5,12.5);
\draw[thick,red] (5,-1) -- (5,-2.5);
\draw[thick,red] (5,-1) -- (3.5,-2.5);
\draw[red] (5,1) node[below] {\small $\omega_2$};
\fill[red] (5,1) circle(3.75pt);

\draw[thick,red] (6,1) -- (12.5,7.5);
\draw[thick,red] (6,1) -- (6,12.5);
\draw[thick,red] (6,-1) -- (6,-2.5);
\draw[thick,red] (6,-1) -- (4.5,-2.5);
\draw[red] (6,1) node[below] {\small $\omega_3$};
\fill[red] (6,1) circle(3.75pt);

\draw[thick,red] (8,2) -- (12.5,6.5);
\draw[thick,red] (8,2) -- (8,12.5);
\draw[thick,red] (8,0) -- (8,-2.5);
\draw[thick,red] (8,0) -- (5.5,-2.5);
\draw[red] (8,2) node[below] {\small $\omega_4$};
\fill[red] (8,2) circle(3.75pt);

\draw[thick,red] (9,0) -- (12.5,3.5);
\draw[thick,red] (9,0) -- (9,12.5);
\draw[thick,red] (9,-2) -- (9,-2.5);
\draw[thick,red] (9,-2) -- (8.5,-2.5);
\draw[red] (9,0) node[below] {\small $\omega_5$};
\fill[red] (9,0) circle(3.75pt);

\end{tikzpicture}\hfill
\begin{tikzpicture}[scale=0.4]
\draw[gray] (-1.5,-1) -- (12.5,-1);
\draw[gray] (0,-2.5) -- (0,12.5);

\draw (11,-0.9) -- (11,-1.1);
\draw (0.1,11) -- (-0.1,11);
\draw (0.1,9) -- (-0.1,9);

\draw[thick,blue] (11,11) -- (12.5,12.5);
\draw[thick,blue] (11,11) -- (11,12.5);
\draw[thick,blue] (11,9) -- (11,-2.5);
\draw[thick,blue] (11,9) -- (-0.5,-2.5);
\draw[blue] (11,11) node[below] {\small $\nu$};
\fill[blue] (11,11) circle(3.75pt);

\draw[thick,red] (3,2) -- (12.5,11.5);
\draw[thick,red] (3,2) -- (3,12.5);
\draw[thick,red] (3,0) -- (3,-2.5);
\draw[thick,red] (3,0) -- (0.5,-2.5);
\draw[red] (3,2) node[below] {\small $\omega_1$};
\fill[red] (3,2) circle(3.75pt);

\draw[thick,red] (5,1) -- (12.5,8.5);
\draw[thick,red] (5,1) -- (5,12.5);
\draw[thick,red] (5,-1) -- (5,-2.5);
\draw[thick,red] (5,-1) -- (3.5,-2.5);
\draw[red] (5,1) node[below] {\small $\omega_2$};
\fill[red] (5,1) circle(3.75pt);

\draw[thick,red] (6,1) -- (12.5,7.5);
\draw[thick,red] (6,1) -- (6,12.5);
\draw[thick,red] (6,-1) -- (6,-2.5);
\draw[thick,red] (6,-1) -- (4.5,-2.5);
\draw[red] (6,1) node[below] {\small $\omega_3$};
\fill[red] (6,1) circle(3.75pt);

\draw[thick,red] (8,2) -- (12.5,6.5);
\draw[thick,red] (8,2) -- (8,12.5);
\draw[thick,red] (8,0) -- (8,-2.5);
\draw[thick,red] (8,0) -- (5.5,-2.5);
\draw[red] (8,2) node[below] {\small $\omega_4$};
\fill[red] (8,2) circle(3.75pt);

\draw[thick,red] (9,0) -- (12.5,3.5);
\draw[thick,red] (9,0) -- (9,12.5);
\draw[thick,red] (9,-2) -- (9,-2.5);
\draw[thick,red] (9,-2) -- (8.5,-2.5);
\draw[red] (9,0) node[below] {\small $\omega_5$};
\fill[red] (9,0) circle(3.75pt);

\fill (3,3) circle(4pt) node[left] {\small $a_0$};
\fill (5,4) circle(4pt) node[left] {\small $a_1$};
\fill (6,2) circle(4pt) node[left] {\small $a_2$};
\fill (8,3) circle(4pt) node[left] {\small $a_3$};
\fill (9,3) circle(4pt) node[right] {\small $a_4$};
\fill (11,2) circle(4pt) node[right] {\small $a_5$};

\end{tikzpicture}
\end{center}
\caption{Identifying shifted generators.}\label{ramp diff three}
\end{figure}

Notice $\tau^{k_1 +1}\omega_1$, as well as $\rho^{j_i-j_{i+1}}\omega_i+\tau^{k_{i+1}-k_i}\omega_{i+1}$ for $1 \leq i \leq n-1$ and $\rho^{j_n +1}\omega_n$ are all elements of $\ker(d)$.
The differential is zero on the first and last of these elements for degree reasons, while $d(\rho^{j_i-j_{i+1}}\omega_i+\tau^{k_{i+1}-k_i}\omega_{i+1})=0$ because we are working mod 2.  Since each of these elements is nonzero in $\tilde{H}^{*,*}(X)$, by exactness we can define $a_0, \dots, a_n$ to be their preimages in $\tilde{H}^{*,*}(Y)$ as in the chart above.  These $a_i$ are free generators and $N \cong \Mtb{a_0, \dots, a_n}$.

Each $a_i$ is uniquely determined because $\iota^*$ is injective in bidegree $|a_i|$.  This follows because in the long exact sequence
\[
\cdots \to \tilde{H}^{|a_i|-(1,0)}(X) \xrightarrow{d} \tilde{H}^{|a_i|}(S^{p,q}) \xrightarrow{\pi^*} \tilde{H}^{|a_i|}(Y) \xrightarrow{\iota^*} \tilde{H}^{|a_i|}(X) \to \cdots
\]
the map $d: \tilde{H}^{|a_i|-(1,0)}(X) \to \tilde{H}^{|a_i|}(S^{p,q})$ is surjective\footnote{This is clear from Figure \ref{ramp diff three} and follows from the image under $d$ of $\rho^{j_i - j_{i+1}-1}\tau \omega_i$ for any $1 \leq i \leq n-1$ and $\rho^{j_n}\tau\omega_n$ since the rank of $\tilde{H}^{*,*}(S^{p,q})$ is at most one in any bidegree.} making $\pi^*$ zero in bidegree $|a_i|$ and hence $\iota^*$ injective.

As the $a_i$ are uniquely determined and we already know $N$ has free generators in the same bidegrees as the $a_i$, each $a_i$ is indeed a free generator.  The other distinguished elements of the kernel that we can lift back to $\tilde{H}^{*,*}(Y)$ are the $\chi_i$.  Since $\tilde{H}^{*,*}(Y) \cong N \oplus \Mtb{b_{n+1}, \dots, b_{m+1}}$, the elements $a_0,\dots,a_n,b_{n+1},\dots,b_{m+1} $ give a free basis for $\tilde{H}^{*,*}(Y)$.

\section{Kronholm's proof}\label{Kronholm proof}

In this section, we discuss the subtle error in Kronholm's proof of the freeness theorem.  The main error is in not completing the inductive step.  The problem appears to arise from the similarity of spectral sequences for two different filtrations of the space.  To describe the issue, we adapt Kronholm's language and notation to match ours.

As a precursor, in Theorem 3.1 of \cite{K}, Kronholm considers attaching a single cell to a complex whose reduced cohomology is a free module with one generator.  He proves, under some assumptions, the cohomology of the newly formed complex is free.

Kronholm's main argument \cite[Theorem 3.2]{K} proceeds by induction on the representation cells. He considers a finite $\Rep(C_2)$-complex $X$ together with the `one-at-a-time' cellular filtration $X_{0} \subseteq \cdots \subseteq X_{k} \subseteq \cdots \subseteq X$, where each $X_{k+1}$ is formed from $X_{k}$ by attaching a single cell. He assumes by induction that $\tilde{H}^{*,*}(X_k)$ is free and aims to show that $\tilde{H}^{*,*}(X_{k+1})$ is free.

He applies a change of basis \cite[Lemma 3.1]{K}, much like our Lemma \ref{basis update}, to $\tilde{H}^{*,*}(X_k)$.  After the change of basis, there is a ramp of generators mapping to the bottom cone of a shifted copy of $\Mt$.  He then mistakenly implies one can reduce to the case of a differential supported by a single free summand, presumably to apply \cite[Theorem 3.1]{K}.  This is not always possible and thus he did not complete the inductive step.

It is likely the error in Kronholm's argument came from conflating two deceptively similar spectral sequences.  Kronholm's proof uses the following spectral sequence (see Proposition 3.1 in \cite{K}).

\begin{prop}\label{spec seq}
Let $X$ be a filtered $C_2$-space.  Then there is a spectral sequence with
\[
E_{1}^{p,q,n} = \tilde{H}^{p,q}\big(X_{n+1},X_{n} \big)
\]
converging to $\tilde{H}^{p,q}(X)$.
\end{prop}

This is a spectral sequence of bigraded $\Mt$-modules. As Kronholm explains, in order to depict the spectral sequence in the plane, we typically use colors to represent the filtration degrees.  The differentials increase topological degree by one as before, but also reaches farther up in the filtration on each page.

For clarity, we use Proposition \ref{spec seq} to define two distinct spectral sequences converging to the cohomology of $X_{k+1}$.
\begin{itemize}
	\item Let $E$ be the spectral sequence for the `one-cell-at-a-time' filtration,
	\[
	X_{0} \subseteq X_1 \subseteq \cdots \subseteq X_{k} \subseteq X_{k+1}.
	\]
	\item Let $\mathcal{E}$ be the spectral sequence for the two-stage filtration $X_k \subseteq X_{k+1}$.
\end{itemize}
In both cases, we write $\nu$ for the free generator corresponding to the last attached cell.  The pictures one would draw of the two spectral sequences are very similar, with some subtle differences.  An example of each is shown in Figure \ref{spectral}.

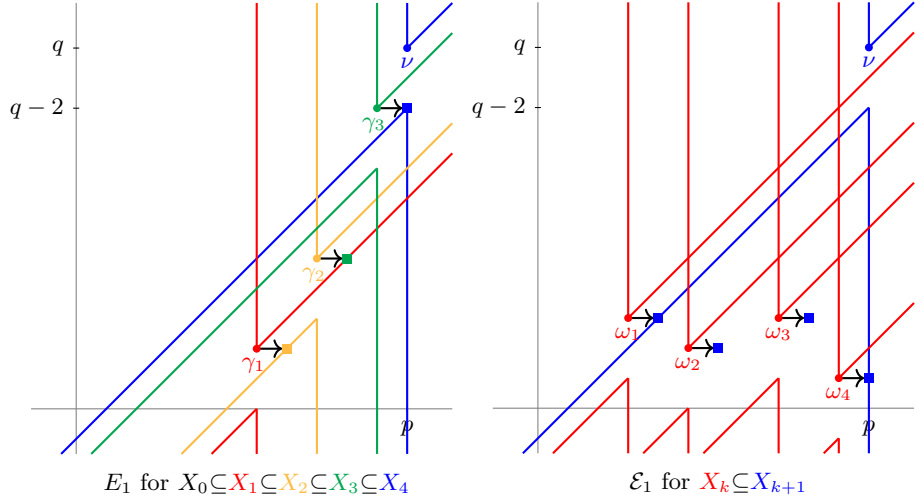
\begin{figure}[ht]
\begin{center}
\begin{tikzpicture}[scale=0.4]
\draw[gray] (-1.5,-1) -- (12.5,-1);
\draw[gray] (0,-2.5) -- (0,12.5);

\draw (11,-0.9) -- (11,-1.1) node[below] {\small $p$};
\draw (0.1,11) -- (-0.1,11) node[left] {\small $q$};
\draw (0.1,9) -- (-0.1,9) node[left] {\small $q-2$};

\draw[->,thick] (6,1) -- (6.85,1);
\draw[->,thick] (8,4) -- (8.85,4);
\draw[->,thick] (10,9) -- (10.85,9);

\draw[thick,blue] (11,11) -- (12.5,12.5);
\draw[thick,blue] (11,11) -- (11,12.5);
\draw[thick,blue] (11,9) -- (11,-2.5);
\draw[thick,blue] (11,9) -- (-0.5,-2.5);
\draw[blue] (11,11) node[below] {\small $\nu$};
\fill[blue] (11,11) circle(3.75pt);

\draw[thick,red] (6,1) -- (12.5,7.5);
\draw[thick,red] (6,1) -- (6,12.5);
\draw[thick,red] (6,-1) -- (6,-2.5);
\draw[thick,red] (6,-1) -- (4.5,-2.5);
\draw[red] (5.9,1) node[below] {\small $\gamma_1$};
\fill[red] (6,1) circle(3.75pt);

\draw[thick,Dandelion] (8,4) -- (12.5,8.5);
\draw[thick,Dandelion] (8,4) -- (8,12.5);
\draw[thick,Dandelion] (8,2) -- (8,-2.5);
\draw[thick,Dandelion] (8,2) -- (3.5,-2.5);
\draw[Dandelion] (7.9,4) node[below] {\small $\gamma_2$};
\fill[Dandelion] (8,4) circle(3.75pt);

\draw[thick,Green] (10,9) -- (12.5,11.5);
\draw[thick,Green] (10,9) -- (10,12.5);
\draw[thick,Green] (10,7) -- (10,-2.5);
\draw[thick,Green] (10,7) -- (0.5,-2.5);
\draw[Green] (9.9,9) node[below] {\small $\gamma_3$};
\fill[Green] (10,9) circle(3.75pt);

\fill[Dandelion] (6.85,0.85) rectangle ++(0.3,0.3);
\fill[Green] (8.85,3.85) rectangle ++(0.3,0.3);
\fill[blue] (10.85,8.85) rectangle ++(0.3,0.3);

\draw (6,-3.5) node {\small $E_1$ for $\color{black}{X_0} \color{black}{\subseteq} \color{red}{X_1} \color{black}{\subseteq} \color{Dandelion}{X_2} \color{black}{\subseteq} \color{Green}{X_3} \color{black}{\subseteq} \color{blue}{X_4} $};

\end{tikzpicture}\hspace{0.01cm}
\begin{tikzpicture}[scale=0.4]
\draw[gray] (-1.5,-1) -- (12.5,-1);
\draw[gray] (0,-2.5) -- (0,12.5);

\draw (11,-0.9) -- (11,-1.1) node[below] {\small $p$};
\draw (0.1,11) -- (-0.1,11) node[left] {\small $q$};
\draw (0.1,9) -- (-0.1,9) node[left] {\small $q-2$};

\draw[->,thick,black] (3,2) -- (3.85,2);
\draw[->,thick,black] (5,1) -- (5.85,1);
\draw[->,thick,black] (8,2) -- (8.85,2);
\draw[->,thick,black] (10,0) -- (10.85,0);

\fill[blue] (3.85,1.85) rectangle ++(0.3,0.3);
\fill[blue] (5.85,0.85) rectangle ++(0.3,0.3);
\fill[blue] (8.85,1.85) rectangle ++(0.3,0.3);
\fill[blue] (10.85,-0.15) rectangle ++(0.3,0.3);

\draw[thick,blue] (11,11) -- (12.5,12.5);
\draw[thick,blue] (11,11) -- (11,12.5);
\draw[thick,blue] (11,9) -- (11,-2.5);
\draw[thick,blue] (11,9) -- (-0.5,-2.5);
\draw[blue] (11,11) node[below] {\small $\nu$};
\fill[blue] (11,11) circle(3.75pt);

\draw[thick,red] (3,2) -- (12.5,11.5);
\draw[thick,red] (3,2) -- (3,12.5);
\draw[thick,red] (3,0) -- (3,-2.5);
\draw[thick,red] (3,0) -- (0.5,-2.5);
\draw[red] (3,2) node[below] {\small $\omega_1$};
\fill[red] (3,2) circle(3.75pt);

\draw[thick,red] (5,1) -- (12.5,8.5);
\draw[thick,red] (5,1) -- (5,12.5);
\draw[thick,red] (5,-1) -- (5,-2.5);
\draw[thick,red] (5,-1) -- (3.5,-2.5);
\draw[red] (5,1) node[below] {\small $\omega_2$};
\fill[red] (5,1) circle(3.75pt);

\draw[thick,red] (8,2) -- (12.5,6.5);
\draw[thick,red] (8,2) -- (8,12.5);
\draw[thick,red] (8,0) -- (8,-2.5);
\draw[thick,red] (8,0) -- (5.5,-2.5);
\draw[red] (8,2) node[below] {\small $\omega_3$};
\fill[red] (8,2) circle(3.75pt);

\draw[thick,red] (10,0) -- (12.5,2.5);
\draw[thick,red] (10,0) -- (10,12.5);
\draw[thick,red] (10,-2) -- (10,-2.5);
\draw[thick,red] (10,-2) -- (9.5,-2.5);
\draw[red] (10,0) node[below] {\small $\omega_4$};
\fill[red] (10,0) circle(3.75pt);

\draw (6,-3.5) node {\small $\mathcal{E}_1$ for $\color{red}{X_k} \color{black}{\subseteq} \color{blue}{X_{k+1}}$};

\end{tikzpicture}
\end{center}
\caption{Examples of spectral sequences for the two filtrations.}\label{spectral}
\end{figure}

On the $E_1$ page of the spectral sequence corresponding to the `one-at-a-time' cellular filtration (shown on the left in Figure \ref{spectral}), the only possible differential to $\Mtb{\nu}$ comes from one filtration degree lower.  However, each intermediate copy of $\Mt$ on the $E_1$ page may both support and receive a nonzero differential.  Such differentials give rise to  non-free ``Jack-O-Lantern'' modules on the $E_2$ page, as studied in \cite{H}.

On the other hand, the differential on the $\mathcal{E}_1$ page for the two-stage filtration is quite different.  As shown on the right in Figure \ref{spectral}, depending on their bidegrees, any copy of $\Mt$ in $\tilde{H}^{*,*}(X_k)$ may support a differential to $\Mtb{\nu}$.  The $\mathcal{E}_2 \cong \mathcal{E}_\infty$ page is typically not free.  Rather, as a consequence of the main theorem, it is an associated graded module of a filtration for a free $\Mt$-module.

Kronholm's proof uses the two-stage spectral sequence $\mathcal{E}$, where a ramp of generators supports a differential.  If the ramp had length one, then \cite[Theorem 3.1]{K} would apply.  It seems likely he was imagining iteratively applying this theorem one filtration degree at a time.  However in $\mathcal{E}$ these maps are simultaneous, so the theorem does not apply.  Switching from $\mathcal{E}$ to $E$ does not help.  On the $E_1$ page, $\Mtb{\nu}$ receives a differential from a single $\Mt$.  However, $E_1$ filtrations can both support and receive differentials.  Moreover, as noted before, $E_2$ is in general not free, so \cite[Theorem 3.1]{K} once again does not apply.

We end with an example demonstrating an $E$ in which a non-free module receives a differential after the $E_1$ page.

\begin{example} Consider the equivariant Grassmannian $\mathbb{P}(\R^{4,1}) = \text{Gr}_1(\R^{4,1})$, using the Schubert cell decompositions for $\text{Gr}_1(\R^{+++-})$ (see \cite{K}, \cite{DGrass}, or \cite{H} for more details).

This $\Rep(C_2)$-complex has cells with bidegrees $(0,0)$, $\color{red} (1,0)$, $\color{Green} (2,0)$, and $\color{blue} (3,3)$. Working with reduced cohomology, we have the one-cell-at-a-time spectral sequence $E$ shown in Figure \ref{jacky}. Note that on the second page, we have a differential from a $\Sigma^{1,0}\Mt$ to an non-free $\Mt$-module, a scenario which is not accounted for in \cite{K}.

\begin{figure}[ht]
\begin{tikzpicture}[scale=0.45]
	\draw (0,-5.5) node {\small $E_1$};
	\draw[gray] (-3.5,0) -- (4.5,0) node[below,black] {\small $p$};
	\draw[gray] (0,-4.5) -- (0,4.5) node[left,black] {\small $q$};
	\foreach \x in {-3,...,-1,1,2,...,4}
		\draw [font=\tiny, gray] (\x cm,2pt) -- (\x cm,-2pt);
	\foreach \y in {-4,...,-1,1,2,...,4}
		\draw [font=\tiny, gray] (2pt,\y cm) -- (-2pt,\y cm);

		\color{black}
		\draw[->,thick] (2,0) --node[below]{$d_1$} (3.15,0);

	\color{red}
	\draw[thick] (4.5,3.5)--(1,0) -- (1,4.5);
	\draw[thick] (-1.5,-4.5)--(1,-2) -- (1,-4.5);
	\fill (1,0) circle(3.75pt);

	\color{Green}
	\draw[thick] (2,4.5)--(2,0) --(4.5,2.5) ;
	\draw[thick] (-.5,-4.5)--(2,-2) -- (2,-4.5);
	\fill[Green] (2,0) circle(3.75pt);

	\color{blue}
	\draw[thick] (3,4.5)--(3,3) --(4.5,4.5) ;
	\draw[thick] (-2.3,-4.5)--(3.2,1) -- (3.2,-4.5);
	\fill (3,3) circle(3.75pt);

\end{tikzpicture}
\begin{tikzpicture}[scale=0.45]
	\draw (0,-5.5) node {\small $E_2$};
	\draw[gray] (-3.5,0) -- (4.5,0);
	\draw[gray] (0,-4.5) -- (0,4.5);
	\foreach \x in {-3,...,-1,1,2,...,4}
		\draw [font=\tiny, gray] (\x cm,2pt) -- (\x cm,-2pt);
	\foreach \y in {-4,...,-1,1,2,...,4}
		\draw [font=\tiny, gray] (2pt,\y cm) -- (-2pt,\y cm);

		\color{black}
		\draw[->,thick] (1,0) --node[below]{$d_2$} (1.95,0);

	\color{red}
	\draw[thick] (4.5,3.5)--(1,0) -- (1,4.5);
	\draw[thick] (-1.5,-4.5)--(1,-2) -- (1,-4.5);
	\fill (1,0) circle(3.75pt);

	\color{Green}
		\draw[thick] (2,4.5)--(2,2)--(3,3)--(3,2.2) --(4.5,3.7) ;
	\draw[thick] (-.5,-4.5)--(2,-2) -- (2,-4.5);

	\color{blue}
	\draw[thick] (3,4.5)--(3,3) --(4.5,4.5) ;
	\draw[thick] (-2.5,-4.5)--(2,0)--(2,-2)--(3,-1) -- (3,-4.5);
\end{tikzpicture}
\begin{tikzpicture}[scale=0.45]
	\draw (0,-5.5) node {\small $E_3=E_\infty$};
	\draw[gray] (-3.5,0) -- (4.5,0);
	\draw[gray] (0,-4.5) -- (0,4.5);
	\foreach \x in {-3,...,-1,1,2,...,4}
		\draw [font=\tiny, gray] (\x cm,2pt) -- (\x cm,-2pt) node[anchor=north] {$\x$};
	\foreach \y in {-4,...,-1,1,2,...,4}
		\draw [font=\tiny, gray] (2pt,\y cm) -- (-2pt,\y cm) node[anchor=east] {$\y$};

	\color{red}
	\draw[thick] (1,4.5)--(1,1)--(2,2)--(2,1) -- (4.5,3.5);
	\draw[thick] (-1.5,-4.5)--(1,-2) -- (1,-4.5);

	\color{Green}
	\draw[thick] (2,4.5)--(2,2)--(3,3)--(3,2.2) --(4.5,3.7) ;
	\draw[thick] (-.5,-4.5)--(2,-2) -- (2,-4.5);

	\color{blue}
	\draw[thick] (3,4.5)--(3,3) --(4.5,4.5) ;
	\draw[thick] (-2.5,-4.5)--(1,-1)--(1,-2)--(2,-1)--(2,-2)--(3,-1) -- (3,-4.5);
\end{tikzpicture}
\caption{$E$ for $\text{Gr}_1(\R^{+++-})$}
\label{jacky}
\end{figure}
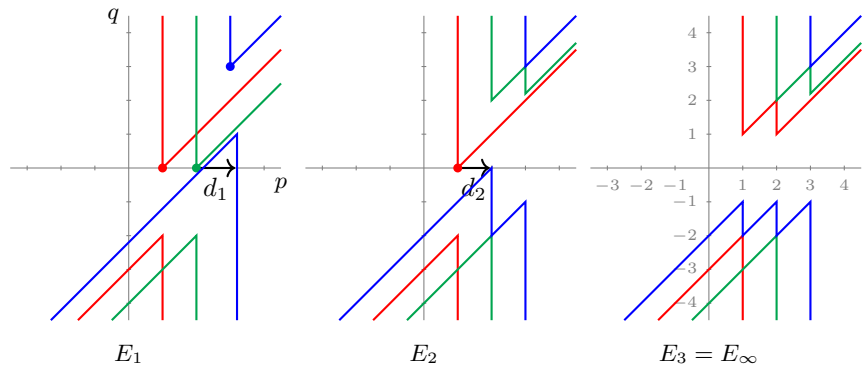

\end{example}

\newpage


\bibliographystyle{abbrv}
\bibliography{refs}
\nocite{*}

\end{document}